\documentclass[10pt]{article}
\usepackage[francais,english]{babel}
\usepackage{amsmath}
\usepackage{amsfonts}
\usepackage{amssymb}
\usepackage{amsthm}
\usepackage{graphics}
\usepackage{amscd}
\usepackage{fullpage}
\usepackage{epsfig}
\usepackage{color}

\usepackage[final]{showkeys}
\usepackage{hyperref}

\newcommand{\Z}{\mathbb{Z}}
\newcommand{\R}{\mathbb{R}}
\newcommand{\N}{\mathbb{N}}

\newcommand{\C}{\mathbb{C}}
\newcommand{\Q}{\mathbb{Q}}

\newcommand{\E}{\mathbb{E}}

\newcommand{\mc}{\mathcal}

\newcommand{\mf}{\mathfrak}
\newcommand{\eps}{\varepsilon}
\newcommand{\ind}{{\bf 1}}
\renewcommand{\P}{\mathbb{P}}

\renewcommand{\H}{\mathbb{H}}

\DeclareMathOperator{\dist}{dist}
\DeclareMathOperator{\Vir}{Vir}

\DeclareMathOperator{\SLE}{SLE}

\title{SLE and Virasoro representations: Fusion}
\author{Julien Dub\'edat\footnote{Partially supported by NSF grant DMS-1005749 and the Alfred P. Sloan Foundation.}}
\newtheorem{thm}{Theorem}

\newtheorem{Thm}[thm]{Theorem}

\newtheorem{Prop}[thm]{Proposition}
\newtheorem{Lem}[thm]{Lemma}

\begin{document}
\maketitle
\begin{abstract}
We continue the study of null-vector equations in relation with partition functions of (systems of) Schramm-Loewner Evolutions (SLEs) by considering the question of fusion. Starting from $n$ commuting SLEs seeded at distinct points, the partition function satisfies $n$ null-vector equations (at level 2). We show how to obtain higher level null-vector equations by coalescing the seeds one by one. As an example, we extend Schramm's formula (for the position of a marked bulk point relatively to a chordal SLE trace) to an arbitrary number of SLE strands.

The argument combines input from representation theory - the study of Verma modules for the Virasoro algebra - with regularity estimates, themselves based on hypoellipticity and stochastic flow arguments.
\end{abstract}

\tableofcontents

\section{Introduction}

The study of Verma modules for the Virasoro algebra, initiated by Kac (e.g. \cite{Kac_Bombay}) and continued in particular in \cite{FeiFuc_Verma} found a spectacular application in the seminal work by Belavin-Polyakov-Zamolodchikov \cite{BPZ} on Conformal Field Theory. The idea is that the correlators of certain primary ``fields" of (scaling limits of) critical two-dimensional statistical mechanical models should satisfy differential equations mapped from special elements of the universal enveloping algebra of the Virasoro algebra - the {\em null} or {\em singular} - vectors, corresponding themselves to exceptional embeddings of Verma modules. These singular vectors are labelled by a pair of integers $(r,s)\in(\N^*)^2$ (for a given central charge); $h_{r,s}$ is the corresponding weight (the scaling dimension of the field, given by \eqref{eq:param}). The role of these special weights in relation with boundary conditions (or states) in BCFT was also observed early, see e.g. \cite{Cardy_fusion}.

Schramm-Loewner Evolutions, introduced by Schramm in \cite{Sch99}, describe (at least conjecturally) the dynamics of interfaces in such models. It was realized early on \cite{BauBer_growth} that the null vectors at level 2 (the simplest non-trivial ones) could be related to the generator of SLE, seen as a diffusion on the configuration space. In this case, corresponding to $(r,s)=(1,2)$ or $(2,1)$, the BPZ differential equations are second-order hypoelliptic PDEs, which are well-understood from a probabilistic point of view. This also gives a very intuitive geometric interpretation of the corresponding field as an ``$\SLE$ creation operator".

In \cite{Ca3,Ca3c,Dub_Comm}, systems of interacting SLEs are constructed and studied (see also \cite{BBK}); in particular their restriction properties \cite{LSW3} involve the weights of Verma modules with higher level degeneracies, with $(r,s)=(n+1,1)$. It is then natural to expect \cite{CarGam_fus} that correlators of ``multiple $\SLE$ creation operators" satisfy higher order (viz. $\geq 3$) BPZ equations, the probabilistic interpretation of which is much less direct.

In the present work we provide some rigorous versions of the BPZ fusion rules within the framework of Virasoro uniformization developed in \cite{Kont_Vir,FriKal,Kont_arbeit,Fri_CFTSLE} (following here the treatment of \cite{Dub_Virloc}). In that context, one considers a section of a suitable line bundle over an extended Teichm\"uller space (keeping track of two marked ``seeds" $X,Y$ on the boundary, additional markings, and formal local coordinates at $X,Y$). Two commuting representations of the Virasoro algebra, corresponding to deformations at each of the seeds, and the assumption is that the section (``partition function") satisfies a general null vector equation at $X$ and a level 2 null vector equation at $Y$. 

We study the leading term of such a partition function as $Y\rightarrow X$, and show that it satisfies itself a null vector equation with an adjacent weight; this is the main result, Theorem \ref{Thm:fus}. As a sample application (Theorem \ref{Thm:Schmult}), we consider a version of Schramm's formula \cite{Sch_percform}, viz. the probability $f_k(\theta)$ that, of $n$ commuting $\SLE$'s from $0$ to $\infty$ in the upper half-plane $\H$, exactly $k$ of the paths pass to the left of a given bulk point $e^{i\theta}$. As expected (and known in some cases, \cite{Sch_percform,CarGam_fus,BelVik_bub}), this satisfies an ODE of order $n+1$, given by the usual BPZ rules, in combination with the Benoit--Saint-Aubin formula \cite{BSA}. While this is a concrete, quantitative statement on SLE, we do not know (and do not expect the existence) of an argument without input from representation theory.

In turn, this can be combined with convergence to $\SLE$ arguments in order to complete the connection with discrete models; examples of discrete observables corresponding to various fusions are discussed in Section \ref{ssec:discrete}.

While some aspects of this work, in particular algebraic ones (Lemma \ref{Lem:algfus} and Section 7), may appear familiar - at least in spirit - to readers knowledgeable of the CFT treatment of fusion, the main contribution of the present article resides in implementing rigorously fusion rules for objects and quantities originating from SLE.

The article (which builds on material detailed in \cite{Dub_Virloc} - we refer the reader to it in particular for background material) is organized as follows. In Section 2, we study a simple example with elementary arguments (and regularity assumptions). Section 3 presents, after some background material, the main algebraic step of the argument. Section 4 justifies the expansion of correlators near the singularity by analytic and probabilistic arguments. The synthesis is operated in Section 5, which contains the main result (Theorem \ref{Thm:fus}). As an application, we justify in Section 6 the extension of Schramm's formula to multiple $\SLE$s. Finally, in Section 7 we show how to obtain BPZ differential equations of order $(r+1)(s+1)$ (in $n+1$ variables) from fusing variables in $r+s$ equations of order 2 (in $n+rs$ variables).

\section{Pairs of SLEs}

We begin with an elementary (and somewhat informal) discussion to the fusion problem in the simplest case, viz. two commuting SLEs in a simply-connected domain aiming at the same target point.

Specifically, $x,y,z_1,\dots,z_n$ are marked points on the real line. We consider two commuting $\SLE_\kappa(\rho)$, $\rho=2$ in the upper half-plane $\H$ , starting from $(x,y)$ and $(y,x)$ respectively and aiming at infinity (see \cite{Dub_Comm} for a discussion of such multiple SLEs, also \cite{BBK,CarGam_fus}). We consider ``martingale observables" of the system of SLEs, eg: the probability that the rightmost SLE hits the interval $(z_1,z_2)$ (say $x<y<z_1<z_2$, $\kappa>4$). These satisfy two second-order PDE's in $n+2$ variables (two seeds $x,y$ and $n$ spectator points $z_1,\dots,z_n$). The goal is to find one third-order PDE in $(n+1)$ variables satisfied by the observable when $x=y$.

\paragraph{First approach.}
Consider the $\SLE_\kappa(\rho)$ started from $(x,y)$, $\rho=2$; $(g_t)$ denotes the associated flow, $(\gamma_t)$ its trace, $X_t=g_t(\gamma_t)$, $Y_t=g_t(y)$. We have the dynamics:
\begin{align*}
dX_t&=\sqrt{\kappa}dB_t+\frac{2dt}{X_t-Y_t}\\
 dY_t&=\frac{2}{Y_t-X_t}dt
 \end{align*}
Set $U=(X+Y)/2$, $V=(Y-X)/2$ (so that $X=U-V$, $Y=U+V$). Then:
\begin{align*}
 dU_t&=\frac{\sqrt{\kappa}}2dB_t\\
dV_t&=-\frac{\sqrt{\kappa}}2dB_t+\frac{dt}{V_t}
\end{align*}
Thus a martingale observable $f(u,v,z_i)$ satisfies ${\mc L}f=0$ where:
$${\mc L}=\frac{\kappa}{8}(\partial_u-\partial_v)^2+\frac{1}{v}\partial_v+\sum\frac{2}{z_i-u+v}\partial_i$$
(here $\partial_i=\frac{\partial}{\partial z_i}$). 

When considering the second SLE, the roles of $x$ and $y$ are exchanged, which corresponds to replacing $v$ with $-v$. Hence an observable also satisfies $\hat{\mc L}f=0$, where:
$$\hat{\mc L}=\frac{\kappa}{8}(\partial_u+\partial_v)^2+\frac{1}{v}\partial_v+\sum\frac{2}{z_i-u-v}\partial_i$$
Remark that, with this change of variable, the commutation relation \cite{Dub_Comm} reads:
$$[{\mc L},\hat{\mc L}]=\frac{1}{v^2}(\hat{\mc L}-{\mc L})$$
It is more symmetric to consider:
\begin{align*}
 {\mc M}_e=\hat{\mc L}+{\mc L}&=\frac{\kappa}4(\partial_{uu}+\partial_{vv})+\frac{2}{v}\partial_v+4\sum\frac{z_i-u}{(z_i-u)^2-v^2}\partial_i\\
{\mc M}_o=\hat{\mc L}-{\mc L}&=\frac\kappa 2\partial_{uv}+4\sum\frac{v}{(z_i-u)^2-v^2}\partial_i
\end{align*}
Expanding around $x=y$, we are now looking for a solution under the form:
$$f(u,v,z_i)=v^\alpha\sum_{n\geq 0}v^nf_n(u,z_i).$$
(The validity of such expansion will be justified in a much more general set-up). The {\em indicial equation} reads:
$$\frac\kappa 2\alpha(\alpha-1)+4\alpha=0$$
i.e. $\alpha\in\{0,1-8/\kappa\}$; it is obtained by examining the coefficient of $v^{\alpha-2}$ in ${\mc L}f$. We consider the case $\alpha=0$ (the other case corresponds to a chordal SLE from $x$ to $y$). Since ${\mc M}_e$ and ${\mc M}_o$ are respectively even and odd under $v\leftrightarrow -v$, we may actually look for a solution of type:
$$f(u,v,z_i)=\sum_{n\geq 0}v^{2n}f_{2n}(u,z_i)$$
The equation ${\mc M}_ef=0$ yields an infinite system of relations on the $f_{2n}$'s (separating by degree in $v$). In particular, in degree 0 we get:
$$\frac{\kappa}4\left(\partial_{uu}f_0+2f_2\right)+4f_2+\sum_i\frac{4}{z_i-u}\partial_if_0=0$$
Similarly, considering the relation ${\mc M}_of=0$ in degree 1, we get:
$$\kappa\partial_uf_2+\sum_i\frac{4}{(z_i-u)^2}\partial_if_0=0$$
We now have two relations between $f_0$ and $f_2$:
\begin{align*}
 -(4+\frac\kappa 2)f_2&=\frac{\kappa}4\partial_{uu}f_0+\sum_i\frac{4}{z_i-u}\partial_if_0\\
-\kappa\partial_uf_2&=\sum_i\frac{4}{(z_i-u)^2}\partial_if_0
\end{align*}
It is now trivial to eliminate $f_2$:
$$\partial_u\left(\frac\kappa 4\partial_{uu}+\sum_i\frac{4}{z_i-u}\partial_i\right)f_0=\frac 4\kappa\sum_i\frac{4+\kappa/2}{(z_i-u)^2}\partial_i f_0$$
We are now looking for translation invariant solutions:
$$f_0(u,z_i)=g_0(z_i-u)$$
where $s_i=z_i-u$. We get $\partial_uf_0=\ell_{-1}g_0$ and $\sum_i (z_i-u)^{n+1}\partial_if_0=-\ell_{n}g_0$, having set $\ell_n=-\sum_i s_i^{n+1}\frac{\partial}{\partial s_i}$. Thus:
$${\mc D}=(\ell_{-1})^3-\frac{16}{\kappa}\ell_{-1}\ell_{-2}+\frac{16}{\kappa^2}(4+\kappa/2)\ell_{-3}$$
satisfies ${\mc D}g_0=0$. Set $\tau=4/\kappa$. Observe that:
$$[\ell_{m},\ell_n]=(m-n)\ell_{m+n}$$
Then:
$${\mc D}=(\ell_{-1})^3-4\tau\ell_{-1}\ell_{-2}+(2\tau+4\tau^2)\ell_{-3}=(\ell_{-1})^3-2\tau(\ell_{-1}\ell_{-2}+\ell_{-2}\ell_{-1})+4\tau^2\ell_{-3}$$
which, as expected, corresponds to the singular vector $\Delta_{3,1}$ (see Section \ref{ssec:singularvec}).\\

\paragraph{Second approach.}
The previous argument exploits the symmetry $x\leftrightarrow y$; let us present another (a bit more formal) argument, which does not rely on this symmetry and is thus easier to generalize.

As before we mark points $x,y,z_i$; we now allow $y,z_1,\dots,z_n$ to have weights $h_y,h_1,\dots,h_n$. We let $v=y-x$, $s_i=z_i-x$ and
\begin{align*}
\ell_n&=-\sum_i s_i^{n+1}\partial_{s_i}-(n+1)h_is_i^n\\
\hat\ell_n&=\ell_n-v^{n+1}\partial_v-(n+1)h_yv^n
\end{align*}
where $\ell_n$ does not account for the marked point $y$. For a perturbation at $y$ (with weight $h_x$ at $x$) we have:
\begin{align*}
\tilde\ell_n&=(-v)^{n+1}(\partial_v-\ell_{-1})-(n+1)h_x(-v)^n
-\left(\sum_i (s_i-v)^{n+1}\partial_{s_i}+(n+1)h_i(s_i-v)^n\right)\\
&=(-v)^{n+1}(\partial_v-\ell_{-1})-(n+1)h_x(-v)^n+\sum_{k\geq 0}(-v)^k{{n+1}\choose{k}}\ell_{n-k}
\end{align*}
(For $n+1<0$, the binomial coefficients are defined by $(1+v)^{n+1}=\sum_{k\geq 0}v^k{{n+1}\choose{k}}$ for $v$ small). 
In particular:
\begin{align*}
\tilde\ell_{-1}&=\partial_v\\
\tilde\ell_{-2}&=-v^{-1}\partial_v+v^{-1}\ell_{-1}+\frac{h_x}{v^2}+\sum_kv^k\ell_{-2-k}
\end{align*}
Consider:
\begin{align*}
{\mc L}&=(\hat\ell_{-1})^2-\tau\hat\ell_{-2}=(\ell_{-1}-\partial_v)^2-\tau(\ell_{-2}-v^{-1}\partial_v+hv^{-2})\\
&=(\partial_{vv}+\tau(v^{-1}\partial_v-h_yv^{-2}))-2\partial_v\ell_{-1}+(\ell_{-1}^2-\tau\ell_{-2})\\
{\mc M}&=(\tilde\ell_{-1})^2-\tau\tilde\ell_{-2}\\
&=(\partial_{vv}+\tau(v^{-1}\partial_v-h_xv^{-2}))-\tau(v^{-1}\ell_{-1}+\ell_{-2}+v\ell_{-3}+\cdots)
\end{align*}
We are looking for $f=v^\alpha\sum_{k\geq 0}v^kf_k$ such that ${\mc L}f={\mc M}f=0$. 

Set $R_.=(\partial_{vv}+\tau(v^{-1}\partial_v-h_.v^{-2}))$ and 
$$r_.(\beta)=v^{2-\beta}R_.v^\beta=\beta(\beta-1)+\tau\beta-\tau h_.$$ 
Examining ${\mc L}f$ in degrees $\alpha-2,\dots,\alpha+1$ yields the relations:
$$\left\{
\begin{array}{lll}
r_y(\alpha)f_0&=0&[L0]\\
r_y(\alpha+1)f_1-2\alpha\ell_{-1}f_0&=0&[L1]\\
r_y(\alpha+2)f_2-2(\alpha+1)\ell_{-1}f_1+((\ell_{-1}^2)-\tau\ell_{-2})f_0&=0&[L2]\\
r_y(\alpha+3)f_3-2(\alpha+2)\ell_{-1}f_2+((\ell_{-1}^2)-\tau\ell_{-2})f_1&=0&[L3]
\end{array}\right.$$
Proceeding similarly for ${\mc M}f$ yields:
$$\left\{
\begin{array}{lll}
r_x(\alpha)f_0&=0&[M0]\\
r_x(\alpha+1)f_1-\tau\ell_{-1}f_0&=0&[M1]\\
r_x(\alpha+2)f_2-\tau\ell_{-1}f_1-\tau\ell_{-2}f_0&=0&[M2]\\
r_x(\alpha+3)f_3-\tau\ell_{-1}f_2-\tau\ell_{-2}f_1-\tau\ell_{-3}f_0&=0&[M3]
\end{array}\right.$$
From $[L0,M0]$, we get $r_x(\alpha)=r_y(\alpha)=0$, that is $h_x=h_y=\tau^{-1}(\alpha+\tau-1)$.\\
From $[L1,M1]$ we get $\alpha=\tau/2$ ($\alpha=2/\kappa$, $\tau h=\frac2\kappa(\frac 6\kappa-1)$) and $2f_1=\ell_{-1}f_0$. It follows that $r_.(\alpha+k)=k(k-1+2\tau)$.\\
Given $2f_1=\ell_{-1}f_0$, $[L2]$ and $[M2]$ are equivalent to:
$$(4\tau+2)f_2=\tau\ell_{-1}f_1+\tau\ell_{-2}f_0=\tau(\frac 12(\ell_{-1})^2+\ell_{-2})f_0$$
Finally, considering $[L3]-[M3]$, one gets:
$$-4\ell_{-1}f_2+(\ell_{-1})^2f_1+\tau\ell_{-3}f_0=0$$
and thus
$$-4\ell_{-1}\left(\tau(\frac 12(\ell_{-1})^2+\ell_{-2})f_0\right)+(4\tau+2)\left((\ell_{-1})^2\frac 12\ell_{-1}f_0+\tau\ell_{-3}f_0\right)=0$$
or $((\ell_{-1})^3-4\tau\ell_{-1}\ell_{-2}+\tau(4\tau+2)\ell_{-3})f_0=0$, as it should.

From this example, we can expect the following general algebraic elimination process: use one of the singular vector equations to write the ``descendants" $f_1,f_2,\dots$ in terms of $f_0$; then plug in these expressions in the other singular vector equation to find, at high enough order in the expansion, a non-trivial relation satisfied by $f_0$. Instead of actually computing this relation, which is impractical in general, we will simply show that it satisfies the characteristic property of a singular vector and use their classification. In order to do this we will need to extend the (finite-dimensional) differential operators $\hat\ell_n$, $\tilde\ell_n$ to commuting Virasoro representations. This is in essence the argument of Lemma \ref{Lem:algfus}.

\section{Algebraic elimination}

\subsection{Singular vectors}\label{ssec:singularvec}

We begin by collecting a few classical facts on the Virasoro algebra, its highest-weight representations and singular vectors; see e.g. \cite{Kac_Bombay,IohKog_Vir} and references therein for a complete account.

The {\em Virasoro algebra} $\Vir$ is the infinite-dimensional Lie algebra 
$$\Vir=\C{\bf c}\oplus \bigoplus_{n\in\Z}\C L_n$$
where ${\bf c}$ is a central element ($[{\bf c},\Vir]=\{0\}$) and for $m,n\in\Z$
$$[L_m,L_n]=(m-n)L_{m+n}+\delta_{m,-n}\frac{m(m^2-1)}{12}{\bf c}$$ 
Set ${\mf h}=\C {\bf c}\oplus \C L_0$, $\Vir^+=\bigoplus_{n>0}\C L_n$, $\Vir^-=\bigoplus_{n<0}\C L_n$. 
Then $\Vir=\Vir^-\oplus {\mf h}\oplus\Vir^+$; $\Vir^\pm,{\mf h}$ are subalgebras; and ${\mf h}$ is abelian.\\
For a Lie algebra ${\mf g}$, we denote by ${\mc U}({\mf g})$ its universal enveloping algebra. We have:
$${\mc U}(\Vir^-)=\bigoplus_{\stackrel{n\geq 0}{0<i_1\leq\dots\leq i_n}}\C L_{-i_k}\dots L_{-i_1}$$
and refer to $\{L_{-i_k}\dots L_{-i_1}: 0<i_1\leq\dots\}$ as the standard basis.

$\Vir$ is a $\Z$-graded Lie algebra, with $\deg(L_n)=n$ for $n\in\Z$ and $\deg({\bf c})=0$; ${\mc U}(\Vir)$ is consequently also $\Z$-graded.

We say that a Virasoro module $V$ has central charge $c\in\C$ if ${\bf c}v=cv$ for all $v\in V$.

In a Virasoro module $V$, a {\em $(c,h)$-highest-weight vector} $v\in V$ is an element s.t. ${\bf c}v=cv$, $L_0v=hv$, $L_nv=0$ if $n>0$. A highest-weight representation is a representation generated by a highest-weight vector $v$ (i.e. spanned by ${\mc U}(\Vir^-)v$).\\
The Verma module $M(c,h)$ is the $(c,h)$-highest-weight module:
$$M(c,h)=\bigoplus_{\stackrel{n\geq 0}{0<i_1\leq\dots\leq i_n}}\C L_{-i_k}\dots L_{-i_1}v$$   
generated by the highest-weight vector $v$. Set
$$M(c,h)_n=\bigoplus_{\stackrel{0<i_1\leq\dots\leq i_k}{i_1+\cdots+i_k=n}}\C L_{-i_k}\dots L_{-i_1}v$$
This is the $(h+n)$-eigenspace of $L_0$; $n$ is the {\em level}.

A {\em singular vector} $w$ in $M(c,h)$ is a highest-weight vector at level $n>0$. If a (non-zero) singular vector exists, it generates a proper submodule (isomorphic to $M(c,h+n)$). It may be written as $w=\Delta v$, $\Delta\in {\mc U}(\Vir^{-})$. Moreover, the coefficient of $L_{-1}^n$ in $\Delta$ has to be nonzero (e.g. Section 5.2.1 in \cite{IohKog_Vir}), so that we may normalize it to be 1. 

We use the parameterization:
\begin{equation}\label{eq:param}
\begin{array}{ll}
c&=13-6(\tau+\tau^{-1})\\
h_{r,s}(\tau)&=\frac{r^2-1}{4}\tau+\frac{1-rs}{2}+\frac{s^2-1}4\tau^{-1}=\frac{(r\tau-s)^2-(\tau-1)^2}{4\tau}
\end{array}
\end{equation}
Trivially $c(\tau)=c(\tau^{-1})$, $h_{r,s}(\tau)=h_{s,r}(\tau^{-1})$.

From the {\em Kac determinant formula}, we know that if $M(c,h)$ is degenerate (contains a singular vector), then $(c,h)=(c(\tau),h_{r,s}(\tau))$ for some $r,s\in\N$, $\tau\in\C^*$.

The general classification of submodules of Verma modules is quite intricate (e.g. \cite{FeiFuc_Verma,IohKog_Vir}). Any highest-weight representation is a quotient of a Verma module, hence the interest of their submodules. For simplicity, we consider the case where $\tau$ is {\em irrational} (case II$_+$ in the Feigin-Fuchs classification). Then for $r,s\in\N^*$, $(c,h)= (c(\tau),h_{r,s}(\tau))$, $M(c,h)$ has a unique proper submodule; it is generated by the highest-weight vector $\Delta_{r,s}v$ at level $rs$ and isomorphic to $M(c,h+rs)$. We can normalize $\Delta_{r,s}=\Delta_{r,s}(\tau)\in {\mc U}(\Vir^-)$ so that the coefficient of $L_{-1}^{rs}$ is 1. Then the other coefficients are Laurent polynomials in $\tau$.

For small $r,s$, $\Delta_{r,s}$ may be evaluated by direct computation (it is a finite-dimensional linear algebra problem). This gives:
\begin{align*}
\Delta_{1,1}(\tau)&=L_{-1}\\
\Delta_{2,1}(\tau)&=L_{-1}^2-\tau L_{-2}\\
\Delta_{3,1}(\tau)&=L_{-1}^3-2\tau(L_{-1}L_{-2}+L_{-2}L_{-1})+4\tau^2L_{-3}
\end{align*}

There is no explicit formula for $\Delta_{r,s}(\tau)$ for general $r,s\in\N$.  When $s=1$, the Benoit--Saint-Aubin formula \cite{BSA} gives:
\begin{equation}\label{eq:BSA}
\Delta_{r,1}(\tau)=\sum_{\stackrel{n_i\geq 1}{n_1+\cdots+n_k=N}}\frac{((r-1)!)^2(-\tau)^{r-k}}{\prod_{i=1}^{k-1}(n_1+\cdots+n_i)(r-n_1-\cdots-n_i)}L_{-n_1}\dots L_{-n_r}
\end{equation}
(and $\Delta_{r,s}(\tau)=\Delta_{s,r}(\tau^{-1})$).

\subsection{Fusion}

Let $t$ be a formal variable, $\alpha\in\R$, $h\in\R$ and $V_{\alpha,h}=\C[[t]][t^{-1}]t^\alpha$ the space of formal series of type:
$$t^\alpha\sum_{k\in\Z}a_kt^k$$
with $\inf\{k:a_k\neq 0\}>-\infty$. Recall that $V_{\alpha,h}$ is a Virasoro module with central charge $0$ and action given by 
$$\ell_n=-t^{n+1}\partial_t-(n+1)ht^n$$ 
(with $\partial_t t^{\alpha+k}=(\alpha+k)t^{\alpha+k-1}$).\\
Given a Virasoro module $W$ with central charge $c$, we may consider $W\bigotimes V_{\alpha,h}$ with action given by 
$$\hat L_n (v\otimes f)=(L_n v)\otimes f+v\otimes (\ell_n f)$$
and the same central charge as $V$. We will omit $\otimes$'s for simplicity and denote by $t^\alpha\sum_{k}v_kt^k$ a generic element of $W\bigotimes V_{\alpha,h}$:
$$\hat L_n(vt^{\alpha+k})=(L_n v)t^{\alpha+k}-(\alpha+k+(n+1)h)v t^{\alpha+k+n}$$
The $\hat L_n$'s may be thought as representing a deformation at a point $x$, keeping track of a nearby point $y=x+t$. We may also consider deformations at $y$. This motivates the introduction of operators:
\begin{align*}
\tilde L_{-1}&=\partial_t\\
\tilde L_{-2}&=-t^{-1}\partial_t+t^{-1}L_{-1}+\frac {\tilde h}{t^2}+\sum_{k\geq 0}t^kL_{-2-k}
\end{align*}
Set $\tilde\Delta_{1,2}=\tilde L_{-1}^2-\tau \tilde L_{-2}$.
To summarize, $L_n$ operates on $W$; $\ell_n$ on $V_{\alpha,h}$; $\hat L_n,\tilde L_n$ on $W\bigotimes V_{\alpha,h}$. The following lemma is a version of the fusion rule: ``$\phi_{r,s}\times \phi_{2,1}=\phi_{r+1,s}+\phi_{r-1,s}$'', and is a conceptual generalization of the elementary elimination arguments of the previous section (second approach). There are some formal similarities with treatments in the CFT literature (see \cite{BdFIZ_singular}, 8.A in \cite{DiF}). Remark however that our arguments do not rely on any operator algebra  assumption \cite{BPZ}.

\begin{Lem}[Fusion]\label{Lem:algfus}
Assume that $w=t^\alpha\sum_{k\geq 0}v_kt^k\in W\bigotimes V_{\alpha,h}$ is a highest-weight vector with weight $h_{r,s}$, $r,s\in\N^*$, and satisfies:
\begin{align*}
\hat\Delta_{r,s}w&=0\\
\tilde \Delta_{2,1}w&=0
\end{align*}
where $h=h_{2,1}$, $\tilde h=h_{r,s}$, $\alpha=h_{r\pm 1,s}-h_{r,s}-h_{2,1}$; and that $\tau$ is irrational. Then $v_0$ is a highest-weight vector in $W$ with weight $h_{r\pm 1,s}$ and satisfies $\Delta_{r\pm 1,s}v_0=0$.
\end{Lem}
\begin{proof}
The proof proceeds in several steps. The idea is to show that from the two conditions algebraic elimination produces a non-trivial condition
$Pv_0=0$, $P\in{\mc U}(\Vir^-)\setminus\{0\}$; and use the classification of Verma (sub)modules to conclude that $\Delta_{r\pm 1,s}v_0=0$.

{\bf Highest-weight condition.} We check that $v_0$ is a $h_{r\pm 1,s}$-highest-weight vector (in $V$). Indeed, since $w$ is a $h_{r,s}$-h.w. vector in $V\bigotimes M_{\alpha,h}$, we have $\hat L_0w=h_{r,s}w$, $\hat L_n w=0$ for $n>0$. Expanding and decomposing by degree, this gives: $(L_0-\alpha-k-h)v_k=h_{r,s}v_k$ and in particular $L_0v_0=h_{r\pm 1,s}v_0$.\\
Examining $\hat L_n w=0$ in degree $0$ gives $L_nv_0=0$ for $n>0$. (The descendants $v_1,\dots$ are obviously not highest-weight). 

{\bf The condition $\tilde \Delta_{2,1}w=0$.}\\
Set $r(\nu)=\nu(\nu-1)+\tau\nu-\tau\tilde h$ so that
$$(\partial_t^2+\tau t^{-1}\partial_t-\tau\tilde ht^{-2})t^{\nu}=r(\nu)t^{\nu-2}$$
We have 
$$\alpha_\pm=h_{r\pm 1,s}-h_{r,s}-h_{2,1}=\frac\tau 4\pm\frac {r\tau-s}2-(\frac 34\tau-\frac 12)=\frac{1-\tau}2\pm \frac {r\tau-s}2
$$
and
\begin{align*}
(\nu-\alpha_+)(\nu-\alpha_-)&=(\nu+\frac{\tau-1}2)^2-\frac{(r\tau-s)^2}4\\
&=\nu^2+(\tau-1)\nu-\tau h_{r,s}=r(\nu)
\end{align*}
We expand $\tilde\Delta_{1,2}w=0$ according to the grading (starting from degree $\alpha-2$):
\begin{align*}
r(\alpha)v_0&=0\\
r(\alpha+1)v_1-\tau L_{-1}v_0&=0\\
r(\alpha+2)v_2-\tau L_{-1}v_1-\tau L_{-2}v_0&=0\\
\vdots\\
r(\alpha+k)v_k-\tau L_{-1}v_{k-1}-\tau\sum_{j=2}^{k}L_{-j}v_{k-j}&=0
\end{align*}
The first equation gives $\alpha=\alpha_\pm$ (provided $v_0$ is nontrivial); this is the indicial equation version of the fusion rule.\\
Since $\tau\notin\Q$, $r-s\tau\notin\Z$. Then $r(\alpha)=0$ and $r(\alpha+k)=k(k\pm(r\tau-s))\neq 0$ for all $k\in\N^*$. Consequently there are elements $R_1,\dots,R_k,\cdots$ of ${\mc U}(\Vir^-)$ s.t. $\tilde \Delta_{1,2}w=0$ iff $v_k=R_kv_0$ for $k\geq 0$ (with $R_0=1$). For instance $R_1=\frac{\tau}{r(\alpha+1)}L_{-1}$, $R_2=\frac{\tau}{r(\alpha+2)}(L_{-1}R_1+L_{-2})$, etc.

{\bf The condition $\hat\Delta_{r,s}w=0$}.\\
Let us consider the first condition, which we may now write as:
$$\hat\Delta_{r,s}\left(t^\alpha\sum_{k\geq 0}t^kR_kv_0\right)=0$$
By decomposing according to degree, this gives elements $P_0,P_1,\dots$ of ${\mc U}(\Vir^-)$ s.t. $P_kv_0=0$, $k\geq 0$:
$$\hat\Delta_{r,s}\left(t^\alpha\sum_{k\geq 0}t^kR_k\right)=t^{\alpha-rs}\sum_{k\geq 0}t^kP_k$$
with $P_k$ at level $k$. We want to verify that $P_k\neq 0$ for some $k$.\\
We are going to focus on the coefficients of $L_{-1}^n$, which are easier to keep track of. Indeed, if $P,Q\in {\mc U}(\Vir^-)$ are homogeneous elements, $P=aL_{-1}^m+\cdots$, $Q=b L_{-1}^n+\cdots$ in the standard basis, then $PQ=ab L_{-1}^{m+n}+\cdots$ in the standard basis, since taking commutators in ${\mc U}(\Vir^-)$ does not produce monomials in $L_{-1}$.

First a trivial induction shows that 
$$R_k=\frac{\tau^k}{r(\alpha+1)\dots r(\alpha+k)}L_{-1}^k+\cdots$$ 
in the standard basis. Write $q(k)=r(\alpha+k)/\tau=k(k+2\alpha+\tau-1)/\tau$, a quadratic polynomial; then $R_k=c_kL_{-1}^k+\cdots$, with $c_k=(q(1)\dots q(k))^{-1}$.

Next we write
$$\hat\Delta_{r,s}=\sum_{i,j,k\geq 0, i+j+k=rs}b_{i,j,k}t^{-i}\partial_t^j L_{-1}^k+\cdots$$
where the remainder does not contain any monomial in $L_{-1}$. By our choice of normalization, $\hat\Delta_{r,s}=\hat L_{-1}^{rs}+\cdots$, so that the leading coefficient $b_{0,0,rs}$ is one.

By contradiction, let us assume that no $P_k$ has a nonzero monomial in $L_{-1}$. This is saying that
$$\left(\sum_{i,j,k\geq 0, i+j+k=rs}b_{i,j,k}t^{-i}\partial_t^j L_{-1}^k\right)\left(t^\alpha\sum_{\ell\geq 0}\frac{t^\ell L_{-1}^\ell}{q(1)\dots q(\ell)}\right)=0$$
in $\C[[t]][t^{-1},\partial_t,L_{-1}]t^\alpha$. For $d=0,\dots,rs$, set 
$$\left(\sum_{i,j\geq 0,i+j=d}b_{i,j,rs-d}t^{-i}\partial_t^j\right)t^{\alpha+\ell+d}=p_d(\ell)t^{\alpha+\ell}$$
so that $p_d$ is a polynomial of degree at most $d$, and $p_0=b_{0,0,rs}=1$. 

By considering the coefficient of $t^{\alpha+\ell}$ in the penultimate expression, we obtain:
$$\frac{p_0(\ell)}{q(1)\dots q(\ell)}+\frac{p_1(\ell)}{q(1)\dots q(\ell+1)}+\cdots+\frac{p_{rs}(\ell)}{q(1)\dots q(\ell+rs)}=0$$
or
$$p_0(\ell)q(\ell+1)\dots q(\ell+rs)+p_1(\ell)q(\ell+2)\dots q(\ell+rs)+\cdots+p_{rs}(\ell)=0$$
This holds for infinitely many $\ell$'s, hence it is a polynomial identity. However the first summand has degree exactly $2rs$ and all other summands have degree at most $2rs-1$, which is the desired contradiction.

{\bf Conclusion.}

Consider the morphism of Viraroso modules $M(c,h_{r\pm 1,s})\rightarrow W$ which maps the generator $v$ of $M(c,h_{r\pm 1,s})$ to $v_0\in W$. Its kernel $N$ is a proper submodule of $M(c,h_{r\pm 1,s})$ (it is not reduced to $\{0\}$ by the previous argument; and we may assume $v_0\neq 0$ otherwise the statement is empty). The degenerate Verma module $M(c,h_{r\pm 1,s})$ has a unique proper submodule (we assumed $\tau\notin\Q$), which is generated by the singular vector $\Delta_{r\pm 1,s}v$. Consequently $\Delta_{r\pm 1,s}v_0=0$ in $W$.
\end{proof}

In order to illustrate the mechanism, we work out a few low level cases (notation as in the proof of Lemma \ref{Lem:algfus}). 

{\bf Case $(2,1)\times (2,1)\rightarrow (3,1)$.}\\
If $\alpha=\alpha_+=h_{3,1}-2h_{2,1}=\frac\tau 2$, we have $r(\alpha+k)=k(k+2\tau-1)$, and by examining $\tilde \Delta_{2,1}w=0$ we see
\begin{align*}
0v_0&=0\\
2\tau v_1&=\tau L_{-1}v_0\\
2(2\tau+1)v_2&=\tau (L_{-1}v_1+L_{-2}v_0)\\
3(2\tau+2)v_3&=\tau (L_{-1}v_2+L_{-2}v_1+L_{-3}v_0)\\
&\vdots
\end{align*}
from which we get 
\begin{align*}
R_0&=1\\
R_1&=\frac 12L_{-1}\\
R_2&=\frac{\tau}{4(2\tau+1)}L_{-1}^2+\frac{\tau}{2(2\tau+1)}L_{-2}\\
R_3&=\frac{\tau}{6(\tau+1)}\left(\frac{\tau}{4(2\tau+1)}L_{-1}^3+\frac{\tau}{2(2\tau+1)}L_{-1}L_{-2}+\frac 12L_{-2}L_{-1}+L_{-3}\right)
\end{align*}
We have
\begin{align*}
\hat \Delta_{2,1}&=\hat L_{-1}^2-\tau\hat L_{-2}\\
&=(L_{-1}-\partial_t)^2-\tau (L_{-2}-t^{-1}\partial_t+h_{2,1}t^{-2})\\
&=(L_{-1}^2-\tau L_{-2})-2L_{-1}\partial_t+(\partial_t^2+\tau t^{-1}\partial_t-\tau h_{2,1}t^{-2})
\end{align*}
Then
\begin{align*}
\hat\Delta_{2,1}\left(t^\alpha\sum_{k\geq 0}t^kR_k\right)&
=t^{\alpha-2}(0)+t^{\alpha-1}(-\tau+2\tau.\frac 12)L_{-1}\\
&+t^{\alpha}\left((L_{-1}^2-\tau L_{-2})-(\frac\tau 2+1)L_{-1}^2+(\frac\tau 2L_{-1}^2+\tau L_{-2})\right)\\
&+t^{\alpha+1}\left[(L_{-1}^2-\tau L_{-2})\frac 12L_{-1}
-2(\frac\tau 2+2)L_{-1}\left(\frac{\tau}{4(2\tau+1)}L_{-1}^2+\frac{\tau}{2(2\tau+1)}L_{-2}\right)\right.\\
&\hphantom{t^{\alpha+1}}\left.
+\tau\left(\frac{\tau}{4(2\tau+1)}L_{-1}^3+\frac{\tau}{2(2\tau+1)}L_{-1}L_{-2}+\frac 12L_{-2}L_{-1}+L_{-3}\right)
\right]+\cdots
\end{align*}
so that $P_0=P_1=P_2=0$ and 
\begin{align*}
(2\tau+1)P_3&=\left(\frac{2\tau+1}2-\frac{\tau(\tau+4)}4+\frac{\tau^2}4
\right)L_{-1}^3
+\left(-\tau(\frac\tau 2+2)+\frac{\tau^2}2\right)L_{-1}L_{-2}+\tau\left(2\tau+1
\right)L_{-3}\\
&=\frac 12L_{-1}^3-2\tau L_{-1}L_{-2}+\tau(2\tau+1)L_{-3}\\
&=\frac 12\left(L_{-1}^3-2\tau(L_{-1}L_{-2}+L_{-2}L_{-1})+4\tau^2 L_{-3}\right)=\frac 12\Delta_{3,1}
\end{align*}
Notice that this explicit computation shows that the lemma holds for any $\tau\neq -\frac 12$ for the fusion $\phi_{2,1}\times \phi_{2,1}=\phi_{3,1}$.

{\bf Case $(1,2)\times (2,1)\rightarrow (2,2)$.}\\
Here $\alpha=\alpha_+=h_{2,2}-h_{1,2}-h_{2,1}=-\frac 12$ ($\alpha_-=\frac 32-\tau$). Then $r(\alpha+k)=k(k+\tau-2)$, 
\begin{align*}
0v_0&=0\\
(\tau-1) v_1&=\tau L_{-1}v_0\\
2\tau v_2&=\tau (L_{-1}v_1+L_{-2}v_0)\\
3(\tau+1)v_3&=\tau (L_{-1}v_2+L_{-2}v_1+L_{-3}v_0)\\
&\vdots
\end{align*}
so that
\begin{align*}
R_0&=1\\
R_1&=\frac \tau{\tau-1}L_{-1}\\
R_2&=\frac{\tau}{2(\tau-1)}L_{-1}^2+\frac 12L_{-2}\\
R_3&=\frac{\tau}{3(\tau+1)}\left(\frac{\tau}{2(\tau-1)}L_{-1}^3+\frac 12L_{-1}L_{-2}+\frac{\tau}{\tau-1}L_{-2}L_{-1}
+L_{-3}\right)
\end{align*}
and
\begin{align*}
\hat \Delta_{1,2}&=\hat L_{-1}^2-\tau^{-1}\hat L_{-2}\\
&=(L_{-1}-\partial_t)^2-\tau^{-1} (L_{-2}-t^{-1}\partial_t+h_{2,1}t^{-2})\\
&=(L_{-1}^2-\tau^{-1} L_{-2})-2L_{-1}\partial_t+(\partial_t^2+\tau^{-1} t^{-1}\partial_t-\tau^{-1} h_{2,1}t^{-2})
\end{align*}
We have $(\partial_t^2+\tau^{-1} t^{-1}\partial_t-\tau^{-1} h_{2,1}t^{-2})t^{\alpha+k}=s(k)t^{\alpha+k-2}$ where
$$s(k)=(k-\frac 12)(k-\frac 32)+\tau^{-1}(k-\frac 12)-(\frac 34-\frac 12\tau^{-1})=k(k+\tau^{-1}-2)
$$
Consider
\begin{align*}
\hat\Delta_{1,2}\left(t^\alpha\sum_{k\geq 0}t^kR_k\right)&
=t^{\alpha-2}(0)+t^{\alpha-1}\left(2\cdot\frac 12+\frac{\tau}{\tau-1}(\tau^{-1}-1)\right)L_{-1}\\
&+t^\alpha\left((L_{-1}^2-\tau^{-1} L_{-2})-2\frac 12\frac\tau{\tau-1} L_{-1}^2+2\tau^{-1}\left(\frac{\tau}{2(\tau-1)}L_{-1}^2+\frac 12L_{-2}\right)
\right)\\
&+t^{\alpha+1}P_3+\cdots
\end{align*}
from which we see $P_0=P_1=P_2=0$,
\begin{align*}
P_3&=(L_{-1}^2-\tau^{-1} L_{-2})\frac{\tau}{\tau-1}L_{-1}-2\frac 32L_{-1}(\frac{\tau}{2(\tau-1)}L_{-1}^2+\frac 12L_{-2})\\
&+3(\tau^{-1}+1)\frac{\tau}{3(\tau+1)}\left(\frac{\tau}{2(\tau-1)}L_{-1}^3
+\frac 12L_{-1}L_{-2}+\frac{\tau}{\tau-1}L_{-2}L_{-1}
+L_{-3}\right)\\
&=\frac{\tau}{\tau-1}(1-\frac 32+\frac 12)L_{-1}^3+\left(-\frac{1}{\tau-1}+\frac{\tau}{\tau-1}\right)L_{-2}L_{-1}+\left(-\frac 32+\frac 12\right) L_{-1}L_{-2}+L_{-3}=0
\end{align*}
and after more (lengthy) computations $P_4\propto\Delta_{2,2}$.

\section{Frobenius series}

Consider a vector-valued ODE with a regular singular point at $0$, i.e.
$$\frac{d}{dx}Y(x)=\frac{A(x)}xY(x)$$
with $A$ matrix-valued and analytic near $0$. Assume that $A(0)$ has a simple spectrum $\lambda_1,\dots,\lambda_n$ and that $\lambda_i-\lambda_j\notin\Z$ if $i\neq j$ (non-resonating condition). Then for $i=1,\dots,n$ the ODE has local solutions of type:
$$x^{\lambda_i}\sum_{k\geq 0}x^kY_k$$
(Frobenius series) with positive radius of convergence; these span the solution space near $0$. See e.g. Chapter XV in \cite{Ince_ODE} for a detailed account.

We want a similar regularity result (once the leading singularity is factored out) in a smooth, hypoelliptic set-up. Consider the following differential operator
\begin{equation}\label{eq:diffop0}
{\mc M}=\frac 12\partial_{xx}+\left(\frac{\tilde\rho}{x-y}+a(x,y,{\bf z})\right)\partial_x+\left(\frac{\tau}{2(y-x)}+b(x,y,{\bf z})\right)\partial_y+\sum_{i=1}^nc_i(x,{\bf z})\partial_{z_i}+\left(-\frac{\tau h}{2(x-y)^2}+\frac{d(x,y,{\bf z})}{y-x}\right)
\end{equation}
where for brevity ${\bf z}=(z_1,\dots,z_n)$. We are interested in the singularity at $y=x$, with $y\geq x$. We assume that the coefficients $a,b,c,d$ are smooth in a relative neighborhood 
$$U=\{(x,y,z): |x-x_0|<\eta, 0\leq y-x<\eta, \|{\bf z}-{\bf z}_0\|<\eta \}$$
of $(x_0,y_0,{\bf z}_0)$ ($\eta>0$ fixed). Set $\Delta=\{(x,y,{\bf z})\in U: x=y\}$, the singular boundary component. We will also denote 
$$Z=\sum_{i=1}^n c_i(x,{\bf z})\partial_{z_i},$$
a vector field.

For example,
$$\frac 12\partial_{xx}+\frac{\tilde\rho}{x-y}\partial_x+\frac{\tau}{2(y-x)}\partial_y+\sum_i\frac{\tau}{2(z_i-x)}\partial_{z_i}-\frac{\tau h}{2(y-x)^2}$$
corresponds to a chordal $\SLE_\kappa(\rho)$ where the $\SLE$ starts from $x$, the force point is at $y$, $z_1,\dots,z_n$ are spectator points, $\tau=4/\kappa$, $\tilde\rho=\rho/\kappa$, and $h$ is the weight at $y$.

Denote by $\alpha_{\pm}$ the roots of the indicial equation 
$$\nu(\nu-1)+(\tau+2\tilde\rho)\nu-\tau h=0$$
with $\alpha_+>\alpha_-$.
The goal of this section is to establish the following:
\begin{Lem}\label{Lem:mainreg}
If $f:U\setminus\Delta\rightarrow\R$ satisfies ${\mc M}f=0$ in $U\setminus\Delta$, $f=O((y-x)^{\alpha_-+\eps})$ for some $\eps>0$, and $\frac 12\partial_x^2+Z$ satisfies the H\"ormander bracket condition \eqref{eq:Hormcond} in $U\setminus\Delta$, then $f=(y-x)^{\alpha_+}g_+$, where $g_+$ has a smooth extension to $U$.
\end{Lem}
More generally, in the non-resonating case $\alpha_+-\alpha_-\notin\N$, we expect that if $f=O((y-x)^{-M})$ for some $M>0$, then
$$f=(y-x)^{\alpha_-}g_-+(y-x)^{\alpha_+}g_+$$
with $g_-,g_+$ extending smoothly to $U$. That would allow to lift the $r\tau-s>0$ assumption in Theorem \ref{Thm:fus}.

The statement of the Lemma is unchanged by replacing ${\mc M}$ with $(y-x)^{-2\tilde\rho}{\mc M}(y-x)^{2\tilde\rho}$ and we will thus assume throughout that $\tilde\rho=0$.

There are two (essentially disjoint) main steps in the proof of Lemma \ref{Lem:mainreg}. The first one (Lemma \ref{Lem:mild}) will be to show that $f$ is {\em mild} in the sense that for every $k\geq 0$, $\nabla^kf=O((y-x)^{-M_k})$ for some $M_k>0$; this will rely on the local approximation of hypoelliptic operators by dilation covariant operators (recalled in Lemma \ref{Lem:hyponil}). The second step, summarized in Lemma \ref{Lem:reg}, is based mostly on stochastic flow arguments.

\subsection{Mildness}

Before we describe the content of the section, let us start with a brief reminder on hypoelliptic operators and H\"ormander's condition (we refer the interested reader to \cite{Stroock_PDE} for a detailed account), in the restricted framework that will be of use to us. Consider two smooth vector fields $X,Y$ (identified with first-order partial differential operators) in an open subset or $\R^d$:
$$X=\sum_{i=1}^d f_i\partial_i,{\rm\ \ \ \ }
Y=\sum_{i=1}^d g_i\partial_i$$
and the second order operator in {\em H\"ormander form}
$${\mc D}=\frac 12X^2+Y$$
(we could also add a smooth potential). This operator is manifestly not elliptic if $d\geq 2$. However, under the {\em H\"ormander condition}:
\begin{equation}\label{eq:Hormcond}
{\rm The\ vector\ fields\ }X,Y{\rm\ and\ their\ iterated\ brackets\ }[X,Y],[X,[X,Y]],\dots
{\rm\ span\ the\ tangent\ space\ at\ every\ point}
\end{equation}
the operator ${\mc D}$ shares key qualitative properties with elliptic operators: if $h$ is a distribution satisfying ${\mc D}h=0$ in an open set, then $h$ is smooth there. More precisely, if ${\mc D}h$ is in $H^s_{loc}$ (Sobolev space), then $h\in H^{s+\delta_0}_{loc}$ where $\delta_0>0$ depends on the number of brackets needed to span the tangent space in \eqref{eq:Hormcond}. For example,
$$\frac 12\partial_{xx}+x\partial_y+x^2\partial_z$$
satisfies \eqref{eq:Hormcond}. We then say that ${\mc D}$ is {\em hypoelliptic}. In \cite{Dub_Virloc}, we checked that $\Delta_{2,1}$ (when written in a trivialization) satisfies \eqref{eq:Hormcond}.

The goal of this subsection is to establish the following

\begin{Lem}\label{Lem:mild}
Let ${\mc M}$ be as in Equation \eqref{eq:diffop0} and assume that $\frac 12\partial_x^2+Z$ satisfies the H\"ormander bracket condition \eqref{eq:Hormcond}. If $f$ is smooth and s.t. ${\mc M}f=0$ and $f=O((y-x)^{-M})$ near a point on the singular boundary, then for all $k\geq 0$ there is $M_k$ s.t. $\nabla^k f=O((y-x)^{-M_k})$ near that point.
\end{Lem}

In the absence of singularity, this is a Harnack estimate \cite{Bony_max}, so we want to obtain some uniformity in this estimate near the singular hypersurface. First we present as a model the case of nilpotent group versions of Bessel processes. Then we recall the local approximation of hypoelliptic operators by translation invariant operators on nilpotent  Lie groups in the smooth case; and adapt that approximation to the type of singular operators under consideration and wrap up the argument.

\subsubsection{Bessel-type processes on nilpotent groups.}

In order to build up some intuition, we describe here a class of model operators, which we can think of as nilpotent group avatars of Bessel processes; and sketch an argument for regularity of solutions up to the boundary. This will not be logically needed later on. For general background on nilpotent Lie groups, see \cite{Goodman_nil}.

Let ${\mf g}$ be a (finite-dimensional) nilpotent Lie algebra generated by ${\bf u},{\bf v}$ (we use boldface for vectors in $\R^d$ as well as Lie algebras); and $G=\exp({\mf g})$ the associated simply-connected Lie group (diffeomorphic to ${\mf g}$ and with group law given by the Campbell-Hausdorff formula). Let ${\mf h}$ be the subalgebra generated by ${\bf v}$ and $[{\mf g},{\mf g}]$, $H$ the corresponding subgroup of $G$, and we assume that ${\mf g}=\R{\bf u}\oplus{\mf h}$; let ${\bf h}_1,\dots,{\bf h}_n$ be a basis of ${\mf h}$. 

Then for $g\in G$, we can write $g=\exp(h)\exp(x{\bf u})$ with $h=y_1{\bf h}_1+\cdots+y_n {\bf h}_n\in{\mf h}$ and $x\in\R$; this defines coordinates $x,y_1,\dots,y_n$ on $G$. 

For $g\in G$, let $R(g)$ be the associated left-invariant vector field on $G$: $(R(g)f)(w)={\frac{d}{dt}}_{|t=0}f(w\exp(tg))$; with our choice of coordinates, $R({\bf u})=\partial_{x}$ and $R(h)x=0$ for all $h\in{\mf h}$. Let us consider the generator
$${\mc G}=\frac 12\partial_{xx}+\frac{\delta-1}{2x}\partial_x+R({\bf v})$$
on $G^+=\{g\in G:x>0\}$, which has a singularity along $H$ and is invariant under left translations by elements of $H$. Since ${\bf u},{\bf v}$ generate ${\mf g}$, it is clear that the H\"ormander bracket condition \eqref{eq:Hormcond} is satisfied and consequently ${\mc G}$ is hypoelliptic on $G^+$. Let us assume that $f:G^+\rightarrow\R$ is bounded and ${\mc G}f$ is compactly supported in $G^+$; we want to extend $f$ smoothly to $H$.

Assume that $(X_t,H_t)_{t\geq 0}$ is a process with values in $\R^+\times H\simeq G^+$ and generator ${\mc G}$; $\P^w$ denotes the law of the process started from $w=(x,h)$. Plainly, $(X_t)$ is a Bessel process, which is transient if $\delta>2$ (which we assume throughout this section). Let
$$({\rm\sf G}u)(x,h)=\E^{x,h}\left(\int_0^\infty u(X_t,H_t)dt\right)$$
if, say, $u:G^+\rightarrow\R$ is continuous and compactly supported (then the LHS is finite, by transience). Then ${\rm\sf G}$ is a positive operator, and arguing as in \cite{Bony_max} we obtain the existence of a Green kernel $g:G^+\times G^+\rightarrow\R$ s.t. 
$$({\rm\sf G}u)(w)=\int g(w,w')u(w')d\mu(w'),$$
where $\mu$ is a Haar measure on $G$. Then ${\mc G}^*g(w,.)=0$ away from $w$ (here ${\mc G}^*$ is the adjoint of ${\mc G}$ w.r.t. $\mu$, which is also hypoelliptic), so that $g$ is smooth away from the diagonal. 

We can start the process from a point on $H$ and $w\mapsto\P^w$ extends continuously (w.r.t. weak convergence) to $H$. It follows that $w\mapsto g(w,.)$ is weakly continuous up to $H$ and consequently (\cite{Bony_max}) $w\mapsto g(w,w')$ extends continuously to $H$, and so do all partial derivatives of $g$ w.r.t. its second argument.

From the invariance of ${\mc G}$ under the left action of $H$ (and the choice of reference measure), it is clear that $g(hw,hw')=g(w,w')$ for all $h\in H$. It follows that the derivatives in the $H$ directions of $g$ w.r.t its first argument also extend continuously to $H$.

Now consider $f$ bounded with ${\mc G}f=u$ compactly supported in $G^+$. Then $f(w)=\int_{G^+}g(w,w')u(w')d\mu(w')$ and the previous argument shows that $R({\bf h}_{i_1})\dots R({\bf h}_{i_k})f$ extends continuously to $H$ for any multi-index $i_1,\dots,i_k$. From here, by simply writing $(\frac 12\partial_{xx}+\frac{\delta-1}{2x}\partial_x)f=-R({\bf v})f$, it is easy to see that $f$ extends smoothly to $H$.

As a concrete example of this situation, consider the Heisenberg group
$$G=H_3(\R)=\left\{\left(   \begin{matrix} %
      1&x&z\\0&1&y\\0&0&1
   \end{matrix}\right):x,y,z\in\R\right\}=\exp({\mf h}_3)$$
with ${\bf u}=\left(\begin{matrix} %
      0&1&0\\0&0&0\\0&0&0
   \end{matrix}\right)$, ${\bf v}=\left(\begin{matrix} %
      0&0&0\\0&0&1\\0&0&0
   \end{matrix}\right)$, $[{\bf u},{\bf v}]=\left(\begin{matrix} %
      0&0&1\\0&0&0\\0&0&0
   \end{matrix}\right)$. Setting $H=\left\{\left(   \begin{matrix} %
      1&0&z\\0&1&y\\0&0&1
   \end{matrix}\right):y,z\in\R\right\}$, we have $G=H\exp(\R{\bf u})$. In these coordinates, we have $R({\bf u})=\partial_x$ and $R({\bf v})=\partial_y+x\partial_z$. The associated Bessel-type process has generator
   $$\frac 12\partial_x^2+\frac{\delta-1}{2x}\partial_x+(\partial_y+x\partial_z)$$
with singularity along $H$.

\subsubsection{Lifting}\label{ssec:lifting}

We start by considering the smooth operator $\frac 12\partial_{xx}+Z$ (excluding for now the singular part $(\frac{\tau}{2(y-x)}+\cdots)\partial_y$). We focus on the neigborhood of a given point on $\Delta$; up to translating variables we may assume $(x,{\bf z})=(0,{\bf 0})$ at that point. We assume that $\frac 12\partial_{xx}+Z$ satisfies the H\"ormander bracket condition, i.e. $\partial_x,Z$ and their iterated brackets up to order $m$ ($m$ a fixed integer) span the tangent space $\R^{n+1}$ at every point in a neighborhood of $(0,{\bf 0})$.

At small scale, an elliptic operator is ``well approximated" by a constant coefficient operator, viz. an operator invariant under translations and homogeneous under (isotropic) dilations. As is well-known \cite{Goodman_nil,RothStein}, hypoelliptic operators at small scale are ``well approximated" by operators invariant under translations corresponding to a non-abelian, nilpotent Lie group structure; and homogeneous under some anisotropic dilations. We review the argument in the present case (following Chapter II of \cite{Goodman_nil}); we will then need to take into account the singular component.

Let us start with a simple example in order to illustrate the difficulty and method. The operator
$$U^2+V=\partial_{xx}+\partial_{z_1}+x\partial_{z_2}+(x^3+z_1)\partial_{z_3}+(x^4+z_2)\partial_{z_4}$$
is hypoelliptic near ${\bf 0}$ (as $U,V,[U,V],[U,[U,[U,[U,V]]]]$ span the tangent space). Changing variable from $z_3$ to $z_3-\frac 12z_1^2$ and scaling $x\leftarrow tx,z_1\leftarrow t^2z_1,z_2\leftarrow t^3z_2,z_3\leftarrow t^5z_3,z_4\leftarrow t^5z_4$, we obtain the scaled operator:
$$\partial_{xx}+\partial_{z_1}+x\partial_{z_2}+x^3\partial_{z_3}+z_2\partial_{z_4}$$
which is homogeneous under these dilations and hypoelliptic. Our goal now is to show that, up to changing variables and possibly adding variables and extending the operator, one obtains an operator that scales to a homogeneous hypoelliptic operator.

Let ${\mf g}_2$ be the free Lie algebra on two generators ${\bf u},{\bf v}$ (its universal enveloping algebra is the algebra of polynomials in two free, non-commuting variables). Consider the grading on ${\mf g}_2=\bigoplus {\mf g}_2^{(n)}$ specified by ${\mf g}_2^{(1)}=\R{\bf u}$, ${\mf g}_2^{(2)}=\R{\bf v}$ (consequently ${\mf g}_2^{(3)}=\R[{\bf u},{\bf v}]$, ${\mf g}_2^{(4)}=\R[{\bf u},[{\bf u},{\bf v}]]$, ${\mf g}_2^{(5)}=\R[{\bf u},[{\bf u},[{\bf u},{\bf v}]]]\oplus\R[{\bf v},[{\bf u},{\bf v}]]$ etc). Let ${\mf g}={\mf g}_2/\bigoplus_{n>2m+1}{\mf g}_{2}^{(n)}$; it is a (finite-dimensional) nilpotent Lie algebra (all iterated brackets of degree $>m$ vanish), and has a basis indexed by well-chosen set of words in ${\bf u},{\bf v}$ (via ${\omega}_1\dots \omega_{k}\rightarrow [\omega_1,[\omega_2,\dots, [\omega_{k-1},\omega_k]\dots]$, $\omega_i\in\{{\bf u},{\bf v}\}$). The algebra ${\mf g}$ is called a free nilpotent Lie algebra of type II in \cite{RothStein}.

Consider the dilation $\delta_t:{\mf g}\rightarrow{\mf g}$ given by $\delta_t(w)=t^kw$ on ${\mf g}^{(k)}$. A function $f$ defined near ${\bf 0}\in{\mf g}$ is of order $k$ if $t^{-k}f(\delta_t w)$ is bounded uniformly in $(t,w)\in (0,
\infty)\times{\mf g}$ close enough to $(0,{\bf 0})$; let $V_k$ denote germs of functions of order $k$. A vector field on ${\mf g}$ defined near ${\bf 0}$ is of order $\leq\ell$ if it maps $V_{k+\ell}$ to $V_k$ for $k\geq 0$. In coordinates, if $w_1,\dots,w_d$ are coordinates on ${\mf g}$ corresponding to a graded basis as described above, $\sum f_i
\partial_{w_i}$ is of order $\leq \ell$ if $f_i$ is of order $\geq \ell+\deg(w_i)$ for $i=1,\dots,d$. Let $L_\ell(U_0)$ denote the space of vector fields of order $\ell$, $U_0$ a neighborhood of the origin in ${\mf g}$. Trivially $[L_\ell(U_0),L_{\ell'}(U_0)]\subset L_{\ell+\ell'}(U_0)$.

Via the exponential map, one may identify ${\mf g}$ with the corresponding simply-connected Lie group $G$; and $w\in{\mf g}$ with the left-invariant vector field $R(w)$:
$$(R(w)f)(g)=\lim_{t\rightarrow 0}\frac{f(ge^{tw})-f(g)}t$$
in such a way that $[R(w_1),R(w_2)]=R([w_1,w_2])$.

Let $L(U)$ be the Lie algebra of (smooth) vector fields in $U$, a neighborhood of ${\bf 0}$ in $\R^{n+1}$. Consider the linear map $\lambda:{\mf g}\rightarrow L(U)$ given by $\lambda([w_1,\dots, [w_{k-1},w_k]\dots])=[\tilde w_1,\dots [\tilde w_{k-1},\tilde w_k]\dots]$ where $w_i={\bf u}$ or ${\bf v}$ and accordingly $\tilde w_{i}=\partial_x$ or $Z$, whenever the degree of that bracket is $\leq m$. In other words (in the terminology of \cite{Goodman_nil}), the map $\lambda$ is the partial Lie algebra homomorphism mapping ${\bf u}$ to $\partial_x$ and ${\bf v}$ to $Z$.

Let $\phi: U_0\rightarrow U, u\mapsto e^{\lambda(u)}{\bf 0}$, where $U_0$ is a small enough neighborhood of ${\bf 0}$ in ${\mf g}$ (here $e^{\lambda(u)}{\bf 0}$ is the element of $\R^{n+1}$ obtained by starting from ${\bf 0}\in\R^{n+1}$ and flowing along the vector field $\lambda(u)\in L(U)$ up to time 1). Then $\phi$ is smooth; and $d\phi({\bf 0})$ is surjective (from the H\"ormander bracket condition $\partial_x,Z$ and their brackets up to order $m$ span the fiber of $L(U)$ at ${\bf 0}$). Up to shrinking $U_0$ and $U$, one may assume that $\phi$ is a submersion.

The key Lifting Theorem (\cite{RothStein}, II.1.3 in \cite{Goodman_nil}) asserts the existence of a linear map $\Lambda: {\mf g}\rightarrow L(U_0)$ s.t.: $\Lambda(w)(f\circ\phi)=(\lambda(w)f)\circ\phi$ for $w\in{\mf g}$, $f\in C^\infty(U)$ (intertwining); if $w\in{\mf g}^{(k)}$, $\Lambda(w)=R(w)\mod L_{k-1}(U_0)$ (approximation); and $\Lambda$ surjective at $0$. For concreteness, we will now express this result in coordinates.

Since $\phi$ is a submersion, we can find coordinates $x,z_1,\dots,z_n,z_{n+1},\dots,z_{n+k}$ s.t. 
$$\phi(x,z_1,\dots,z_n,z_{n+1},\dots,z_{n+k})=(x,z_1,\dots,z_n)$$
locally near ${\bf 0}$. Set $\tilde X=\Lambda({\bf u})$, $\tilde Z=\Lambda({\bf v})$, vector fields on $U_0$. By the intertwining condition we have $\tilde X=\partial_x+\sum_{j>n} f_j\partial_{z_j}$ and $\tilde Z=Z+\sum_{j>n} g_j\partial_{z_j}$. 

By the approximation condition we have $\tilde X=R({\bf u})\mod L_0(U_0)$; and $\tilde Z=R({\bf v})\mod L_1(U_0)$. Moreover $R({\bf u})\in L_1(U_0)$, $R({\bf v})\in L_2(U_0)$ (\cite{Goodman_nil}, II.2.2). 
We identify ${\mf g}$ with tangent vectors at ${\bf 0}$ on $U_0$. Let $(w_0,w_1,\dots,w_N)$ be coordinates on ${\mf g}$ relative to a graded basis, so that $\partial_0$ spans ${\mf g}^{(1)}$, $\partial_{1}$ spans ${\mf g}^{(2)}$, $\partial_4,\partial_5$ span ${\mf g}_2^{(5)}$ etc; this assigns a degree to each variable (e.g. $\deg(w_4)=\deg(w_5)=5$). In these coordinates (and up to scaling $w_0$ and $w_1$) we have
\begin{align*}
\tilde X&=g_0\partial_0+\sum_{i>0} g_i\partial_{i}\\
\tilde Z&=f\partial_0+\sum_{i>0} h_i\partial_{i}
\end{align*}
with $g_0({\bf 0})=h_1({\bf 0})=1$, $g_i\in V_{\deg(w_i)-1}$, $h_i\in V_{\deg(w_i)-2}$.

For a vector field $Y$, let $\delta_t Y$ be the vector field s.t. $((\delta_t Y)f)\circ\delta_t=Y(f\circ\delta_t)$. We say that $Y$ is homogeneous of degree $k$ if $\delta_tY=t^k Y$; in coordinates $Y=\sum e_i\partial_{w_i}$ with $e_i$ homogeneous of degree $k+\deg(w_i)$. More generally, one can define homogeneous differential operators and their degrees in such a way that the degree of the composition of operators is the sum of the degrees. In particular $R(w)$ is homogeneous of degree $k$ if $w\in {\mf g}^{(k)}$.

We may rewrite the approximation condition as $t^{-1}\delta_t\tilde X=R({\bf u})+O(t)$, $t^{-2}\delta_t\tilde Z=R({\bf v})+O(t)$ (e.g. in the sense of uniform convergence of all derivatives of coefficients on compact subsets) and
$$t^{-2}\delta_t \left(\frac 12\tilde X^2+\tilde Z\right)=\frac 12R({\bf u})^2+R({\bf v})+O(t)$$

The homogeneous, translation invariant operator $\frac 12R({\bf u})^2+R({\bf v})$ satisfies the H\"ormander bracket condition (since ${\mf g}$ is generated by ${\bf u},{\bf v}$). In summary, we have the following (\cite{Goodman_nil,RothStein}):

\begin{Lem}\label{Lem:hyponil}
Let ${\mc M}_0=\frac 12\partial_{xx}+\sum_{i=1}^n c(x,{\bf z})\partial_{z_i}=\frac 12X^2+Z$ be s.t. $X,Z$ and their iterated brackets span the tangent space at ${\bf 0}$. Then there is $N\geq n$, a submersion $\phi:\R^{N+1}\rightarrow\R^{n+1}$ defined in a neighborhood $U_0$ of ${\bf 0}$ s.t. $\phi({\bf 0})={\bf 0}$; and an operator $\tilde{\mc M}_0=\frac 12\tilde X^2+\tilde Z$ hypoelliptic near $0$ s.t. $\tilde{\mc M}_0(f\circ\phi)=({\mc M}_0f)\circ\phi$. 
Moreover one can assign a degree $\deg(w_i)$ to each coordinate $w_i$ of $\R^{N+1}$ with $\deg(w_0)=1$ in such a way that $t^{-2}\delta_t\tilde{\mc M}_0={\mc M}_0^s+O(t)$, where
${\mc M}_0^s=\frac 12R({\bf u})^2+ R({\bf v})$ is homogeneous, hypoelliptic, and left-invariant under translations relative to a nilpotent Lie group structure on $\R^{N+1}$.
\end{Lem}

We now want to adapt this classical result on local models of hypoelliptic operators with smooth coefficients to the present situation, viz. in the presence of a singular hyperplane.

We start by extending the Lie algebra ${\mf g}$ to the graded, nilpotent Lie algebra $\hat{\mf g}={\mf g}\bigoplus\R{\bf c}$ where ${\bf c}$ has degree 1 and is central (viz. $[{\bf c},w]=[{\bf u},w]$ for all $w\in{\mf g}$); $\hat{\mf g}$ is a trivial central extension of ${\mf g}$ and exponentiates to a Lie group $\hat G\simeq G\times\R$, into which $G$ is naturally embedded. A generic element of $\hat{\mf g}$ is written as $w_0{\bf u}+y{\bf c}+w_1{\bf v}+\cdots$. 

We extend the submersion $\phi:(w_0,w_1,\dots,w_N)\mapsto (x,z_1,\dots,z_n)=(\phi_0(w_0,\dots),\phi_1(w_0,\dots,),\dots)$ by 
$$\hat\phi:(w_0,y,w_1,\dots,w_N)\mapsto (x,y,z_1,\dots,z_n)=(\phi_0(w_0,w_1,\dots),y,\phi_1(w_0,w_1,\dots),\dots)$$
We write $\hat X=\tilde X+0\partial_{w'_0}$ ($\hat X$ is a vector field defined near the origin of $\R^{N+2}$; $\tilde X$ is a vector field on $\R^{N+1}$) so that $\hat X$ and $\partial_x$ are intertwined by $\hat\phi$.

Set $Y=\left(\frac{\tau}{2(y-x)}+b(x,y,{\bf z})\right)\partial_y$ and $\hat Y=\left(\frac{\tau}{2(y-\phi_0(w_0,\dots))}+b\circ\hat\phi\right)\partial_y$; then $Y$ and $\hat Y$ are intertwined by $\hat\phi$. Similarly, set $v=\left(-\frac{\tau h}{(x-y)^2}+\frac{d(x,y,{\bf z})}{y-x}\right)$ and $\hat v=v\circ\hat\phi$. Then
$${\mc M}=\frac 12\partial_x^2+Y+Z+v{\rm\ and\ }\hat{\mc M}=\frac 12\hat X^2+\hat Y+\hat Z+\hat v$$
are intertwined by $\hat\phi$.

We consider the behavior of $\hat{\mc M}$ under scaling, where now
$$\delta_t(w_0,y,w_1,\dots)=(tw_0,ty,t^2w_1,\dots,t^{\deg(w_i)}w_i,\dots)$$
By Lemma \ref{Lem:hyponil}, we already know that $\hat{\mc M}_0=\frac 12\hat X^2+\hat Z$ scales to $\hat{\mc M}_0^s=\frac 12R({\bf u})^2+R({\bf v})$. We have $\hat X\phi_0=\hat X(x\circ\hat\phi)=(\partial_xx)\circ\hat\phi=1$ and $\hat X=\partial_{w_0}+o(1)$ with our choice of coordinates. Then 
$$\phi_0\circ\delta_t=tw_0+O(t^2)$$
Consequently, $t^{-2}\delta_t\hat{\mc M}=\hat{\mc M}^s+O(t)$ where
$$\hat{\mc M}^s=\frac 12R({\bf u})^2+R({\bf v})+\frac{\tau}{2(y-w_0)}\partial_y-\frac{\tau h}{2(y-w_0)^2}$$

Manifestly $\partial_y$ is invariant by left translations by elements of $G$ (and so are $R({\bf u}),R({\bf v})$). Let $G^{\geq 2}=\exp({\mf g}^{(\geq 2)})$ be the subgroup generated by elements of degree ${\geq 2}$. Then $R(w)w_0=0$ for all $w\in G^{\geq 2}$ (e.g. from the Campbell-Hausdorff formula); and $R({\bf u})w_0=1$. Consequently $\hat{\mc M}^s$ is invariant under left translations by elements of type $(y,y,w_1,\dots)$, which generate a codimension one subgroup $G_0$. Together with dilations, these translations operate transitively on $\hat G\setminus G_0$.

Finally we observe that $\hat{\mc M}^s$ satisfies itself the H\"ormander bracket condition. Indeed $R({\bf u}),R({\bf v})$ and their brackets generate ${\mf g}$; $\hat Y^s=\frac{\tau}{2(y-w_0)}\partial_y$ commutes with ${\mf g}^{(\geq 2)}$. The iterated brackets of $R({\bf u})$ and $\hat Y^s$ are of type:
$$[R({\bf u}),[R({\bf u}),\dots, [R({\bf u}),\hat Y^s]\dots]\propto\frac{1}{(y-w_0)^{k+1}}\partial_y$$
so that for $k$ large enough $[R({\bf u}),[R({\bf u}),\dots, [R({\bf u}),R({\bf v})+\hat Y^s]\dots]\propto\frac{1}{(y-w_0)^{k+1}}\partial_y$.
 
Without loss of generality, we will assume in what follows that ${\mc M}$ is of the same type as $\hat{\mc M}$, i.e. scales to a homogeneous hypoelliptic operator invariant under a codimension 1 group of translations (changing $n$ to $N$ and renaming $(w_0,w_0',w_1,\dots)$ to $(x,y,z_1,\dots)$ if necessary).

\subsubsection{Harnack estimate}

We consider an operator ${\mc M}$ as obtained in Section \ref{ssec:lifting}, which is defined in a space diffeomorphic to $(0,\infty)\times\R^{n+1}$, with a singularity along the boundary $S$.

The goal here is to establish the following: if $f$ is s.t. ${\mc M}f=0$ in a neighborhood $U$ of a point of the singular boundary $S$ (which we may choose to be ${\bf 0}$) and if $f$ has polynomial growth, viz.
$$f=O(\dist(.,S)^{-M})$$
for some $M>0$, then $f$ is mild in the sense that for all $k\geq 0$, we have
$$\nabla^k f=O(\dist(.,S)^{-M_k})$$
for some $M_k>0$. 

From the classical hypoelliptic Harnack inequality of Bony \cite{Bony_max}, we have that if $K,K'$ are compact neighborhoods in $U$, $K\subset\subset K'\subset\subset U$, then there is $c=c_k(K,K')$ s.t.
\begin{equation}\label{eq:Bony_Harnack}
\sup_{K}|\nabla^k f|\leq c_k(K,K')\sup_{K'} |f|
\end{equation}
for all $f$ with ${\mc M}f=0$ in $K'$. 

The claim follows easily for the homogeneous, translation invariant operator ${\mc M}^s$ (in that case, $S=G_0$, a codimension 1 subgroup of $\hat G$). Indeed, let $V\subset\subset U$ be a neighborhood of ${\bf 0}$; for fixed compact neighborhoods $K,K'$ as above, we can find a neighborhood of ${\bf 0}$ contained in $V$ and covered by translates and dilations of $K$: $\cup_{\alpha\in A}\delta_{t_\alpha}(h_\alpha K)$, with the $t_\alpha$'s positive and bounded and the $h_\alpha$'s bounded in $G_0$, and $\cup_{\alpha\in A}\delta_{t_\alpha}(h_\alpha K')\subset U$. We have 
\begin{equation}\label{eq:harnackscal}
\sup_{\delta_t(hK)}|\nabla^k f|\leq C t^{-km}\sup_K|\nabla^k(f\circ\delta_t\circ\tau_h)|
\end{equation}
with $m=\max_i\deg(w_i)$ and $h$ close to the identity in $G_0$ ($\tau_h$ denotes the left translation by $h$, which commutes with ${\mc M}^s$). It follows directly that if $f=O(\dist(.,G_0)^{-M})$ and ${\mc M}^sf=0$, then $\nabla^k f=O(\dist(.,G_0)^{-M-mk})$.
 
 Going back now to a general ${\mc M}$, we need a small refinement of \eqref{eq:Bony_Harnack}: if ${\mc L}_t$ is a family of hypoelliptic operators satisfying H\"ormander's condition and depending smoothly on $t$ (a vector of parameters), for $K\subset\subset K'$ compact sets and $k\geq 0$ and $f$ s.t. ${\mc L}_tf=0$ in $K'$, we have
 \begin{equation}\label{eq:paramharnack}
 \sup_{K}|\nabla^k f|\leq c(k,K,K',t)\sup_{K'} |f|
 \end{equation}
 with $t\mapsto c(k,K,K',t)$ locally bounded.

To see this, we reason by  contradiction and assume that $(f_n)$ is a bounded sequence of smooth functions in $K'$ and $(w_n)$ is a sequence in $K$ s.t. ${\mc L}_nf_n=0$ in $K'$, where ${\mc L}_n={\mc L}_{t_n}$, $t_n\rightarrow 0$, $w_n\rightarrow w$, and $|\nabla^k f_n(w_n)|\nearrow\infty$. If $\phi$ is a smooth cutoff function supported in $K'$ s.t. $\phi=1$ in $K$, we may assume (up to extracting a subsequence) that $\phi f_n$ and its first-order derivatives converge in the Sobolev space $H_s(\R^{n+2})$ for some $s<0$. In order to obtain the contradiction it suffices to show that the convergence occurs in $H_{s'}$ for arbitrarily large $s'>0$, and consequently in $C_k$ for any $k\geq 0$.

We shall use the H\"ormander subelliptic estimate (\cite{Horm_hypo}, Theorem 7.4.17 in \cite{Stroock_PDE})
$$\|\phi_1u\|_{s+\delta_0}\leq c_s\left(\|\phi_2{\mc L}_nu\|_{s}+\|\phi_2u\|_{s}\right)$$
for $K_1\subset\subset K_2$, where $\delta_0>0$ and $\phi_1,\phi_2\in C^\infty_c$ are s.t. $\phi_2=1$ on an open neighborhood of $supp(\phi_1)$, $\phi_1=1$ on $K_1$. Here $\|.\|_s$ denotes the norm of the Sobolev space $H_s$. %
A careful examination of the argument shows that $\delta_0$ and $c_s$ may be chosen independently of $n$ for $n$ large enough: indeed $\delta_0$ depends on the maximal number of brackets needed to generate the tangent space at every point in $supp(\phi_1)$; $c_s$ depends on the Sobolev norms of the coefficients of ${\mc L}_n$ (with cut-offs).

We apply the estimate to $u=f-f_n$, $f$ the limit of $\phi f_n$. Clearly $f$ is a weak solution of ${\mc L}_0f=0$ in $K_1$, and thus (H\"ormander) is smooth. We have ${\mc L}_n(f-f_n)=({\mc L}_n-{\mc L}_0)f$ in $K'$ and consequently $\|\phi_2{\mc L_n}(f-f_n)\|_{s'}=O(t_n)$ for any $s'$. By the subelliptic estimate, this shows that $\phi_1f_n$ converges in $H_{s+\delta_0}$; by bootstrapping the argument (taking cutoff functions with slightly smaller support at each step), one obtains convergence in $H_{s'}$ for arbitrarily large $s'$, as desired; this gives \eqref{eq:paramharnack}.

We go back now to the operator ${\mc M}$ and consider the parametric family of operators ${\mc M}_t=t^{-2}\delta_t{\mc M}$ with ${\mc M}_0={\mc M}^s$. By the earlier scaling argument \eqref{eq:harnackscal} together with \eqref{eq:paramharnack}, we obtain:
$$|\nabla^k f|(x)=O(|x|^{-M_k})$$ 
if ${\mc M}f=0$ near $0$, $f=O(\dist(.,S)^{-M})$ and $x\in\cup_{t>0}\delta_tK$, $K$ a fixed compact set. This gives the desired estimate when one approaches the singularity ${\bf 0}$ in a ``cone" $\cup_{t>0}\delta_tK$ (this is not a cone in the usual affine sense, since the $\delta_t$'s are not standard dilations).

Finally  we argue that this last estimate is locally uniform in the apex. Indeed, take a point ${\bf v}$ nearby ${\bf 0}$ on the singular hypersurface. We can go through the lifting argument, centered now at ${\bf v}$. The resulting scaled operator is still ${\mc M}^s$ (it depends solely on the number of brackets needed to generate the tangent space and the value of $\tau$). More precisely, we have a diffeomorphism $\phi_{\bf v}$ mapping ${\bf 0}$ to ${\bf v}$ s.t. $({\phi_{\bf v}})^*{\mc M}$ has the same type of scaling as ${\mc M}$, viz. ${\mc M}_{{\bf v},t}=\delta_t(\phi_{\bf v})^*{\mc M}=t^2{\mc M}^s+O(t^3)$. From the construction of the lifting theorem, it is clear that $\phi_{\bf v}$ can be chosen smoothly in ${\bf v}$ for ${\bf v}$ near ${\bf 0}$. Consequently the family $({\mc M}_{{\bf v},t})_{{\bf v},t}$ is smooth (away from the singular hypersurface). By applying \eqref{eq:paramharnack} to the family ${\mc M}_{{\bf v},t}$ and covering a neighborhood of ${\bf 0}$ by subsets of type $\phi_{\bf v}(\delta_tK)$, we obtain
$$|\nabla^k f|(x)=O(|x|^{-M_k})$$ 
if ${\mc M}f=0$ near $0$, $f=O(\dist(.,S)^{-M})$ and $x$ is near $0$. This concludes the proof of Lemma \ref{Lem:mild}.

\subsection{Regularity}

In this section the goal is to establish the following

\begin{Lem}\label{Lem:reg}
Let ${\mc M}$ be as in \eqref{eq:diffop0}. If $f$ is s.t. ${\mc M}f=0$ in $U\setminus\Delta$; $f$ is smooth and mild; and $f=O(\dist(.,\Delta)^{\alpha_-+\eps})$ for some $\eps>0$, then $f$ extends to a smooth function on $U$.
\end{Lem}

The argument is mostly based on the analysis of a stochastic flow, itself constructed from a standard Bessel flow; and involves hypoellipticity only through the assumption that $f$ is smooth and mild.

\subsubsection{Conjugation}

We consider an operator ${\mc M}$ of type
\begin{align*}
{\mc M}&=\frac 12\partial_{xx}+a(x,y,{\bf z})\partial_x+\left(\frac{\tau}{2(y-x)}+b(x,y,{\bf z})\right)\partial_y+\sum_{i=1}^nc(x,{\bf z})\partial_{z_i}+\left(-\frac{\tau h}{(x-y)^2}+\frac{d(x,y,{\bf z})}{y-x}\right)
\end{align*}
as in \eqref{eq:diffop0}.

We proceed with factoring out the leading singularity. Let $\alpha=\alpha_{\pm}$ be a root of the indicial equation
$$\frac 12\alpha(\alpha-1)+\frac\tau 2\alpha-\tau h=0$$
and ${\mc M}_0=(y-x)^{-\alpha}{\mc M}(y-x)^{\alpha}$. Then
$${\mc M}_0=\frac 12\partial_{xx}+\left(\frac{\alpha}{x-y}+a(x,y,{\bf z})\right)\partial_x+
\left(\frac{\tau/2}{y-x}+b(x,y,{\bf z})\right)\partial_y+\sum_i c_i(x,{\bf z})\partial_{z_i}+\frac{(\alpha(a-b)+d)(x,y,{\bf z})}{y-x}
$$
It will be convenient to write the singular part in a standard Bessel form. First we change variables from $(x,y,{\bf z})$ to $(r,y,{\bf z})$, where $r=y-x$. Then 
$${\mc M}_0=\frac 12\partial_{rr}+\left(\frac{\delta-1}{2r}+(b-a)(r,y,{\bf z})\right)\partial_r+
\left(\frac{\tau/2}{r}+b(r,y,{\bf z})\right)\partial_y+\sum_i c_i(r,{\bf z})\partial_{z_i}+\frac{(\alpha(a-b)+d)(r,y,{\bf z})}{r}
$$
where $\frac{\delta-1}2=\frac\tau 2+\alpha$. If $\alpha=\alpha_+>\alpha_-$, since $\alpha_++\alpha_-=1-\tau$, we have $\delta>2$. %

Setting $h(r,y,{\bf z})=\exp\left(2\int_0^r (a-b)(s,y,{\bf z})ds\right)$, we get 
\begin{equation}\label{eq:diffopL}
{\mc L}\stackrel{def}{=}h^{-1}{\mc M}_0h=
\frac 12\partial_{rr}+\frac{\delta-1}{2r}\partial_r+
\left(\frac{\tau}{2r}+b(r,y,{\bf z})\right)\partial_y+\sum_i c_i(r,{\bf z})\partial_{z_i}+\frac{\tilde d(r,y,{\bf z})}{r}
\end{equation}
After this shift we have $\alpha_-=2-\delta<0$, $\alpha_+=0$.

\subsubsection{Flow}

We write ${\mc L}={\mc L}_0+\frac{\tilde d(r,y,{\bf z})}{r}$, where ${\mc L}_01=0$. We shall construct a process (more precisely, a flow) corresponding to the generator ${\mc L}_0$. Throughout it is crucial that $\delta>2$. We start by listing a few facts about standard Bessel flows - see \cite{RY} and \cite{Vostr_Besflow} for a detailed analysis.

Let $(B_t)_{t\geq 0}$ be a standard Brownian motion (BM) and consider the stochastic flow defined by:
\begin{equation}\label{eq:Besflow}
\left\{
\begin{array}{lll}
dR^{r}_t  &=dB_t+\frac{\delta-1}{2R^{r}_t}dt&\forall t\geq 0,r\geq 0 \\
R^r_0&=r&{\forall r\geq 0}
\end{array}
\right.
\end{equation}
Then
\begin{enumerate}
\item for fixed $r\geq 0$, \eqref{eq:Besflow} has a unique strong solution, whose law is that of a ${\rm Bes}(\delta)$ ($\delta$-dimensional Bessel process) started from $r$. It is positive for positive times and transient.
\item $(r,t)\mapsto R^r_t$ has a bicontinuous version $[0,\infty)^2\rightarrow [0,\infty)$.
\item If $r\leq r'$, a.s. $R^r_t\leq R^{r'}_t$ for all $t\geq 0$.
\item $(\lambda^{-1} R^{\lambda r}_{\lambda^2 t})_{r,t\geq 0}$ has the same law as $(R^r_t)_{r,t\geq 0}$ ($\lambda>0$ a constant)
\item a.s. for each $t>0,r\geq 0$, $\int_0^tds/R^r_s<\infty$.
\item For fixed $t$, $r\mapsto R^r_t$ is a.s. smooth on $(0,\infty)$, $C^2$ and 1-Lipschitz on $[0,\infty)$.
\end{enumerate} 
Remark that in 6., one may {\em not} replace $C^2$ by, say, $C^3$ for general $\delta>2$.

Now let us consider the stochastic flow defined (at least for short times) by
\begin{equation}\label{eq:SDEflow}
\left\{
\begin{array}{ll}
dR^{\bf v}_t  &=dB_t+\frac{\delta-1}{2R^{\bf v}_t}dt \\
dY^{\bf v}_t  &=\left(\frac{\tau}{2R^{\bf v}_t}+b({\bf V}^{\bf v}_t)\right)dt \\
d{\bf Z}^{\bf v}_t&={\bf c}({\bf V}^{\bf v}_t)dt
\end{array}
\right.
\end{equation}
where ${\bf V}_t=(R_t,Y_t,{\bf Z}_t)$, $V_0={\bf v}$. Note that the evolution of $R$ does not depend on $Y,{\bf Z}$. Thus the flow may be constructed in the following way: $(R^r_t)_{r,t}$ is a standard Bessel flow \eqref{eq:Besflow}, defined for all $r,t\geq 0$. We are in the transient case: $\delta>2$; in particular $\int_0^t\frac{ds}{R^r_s}$ is a.s. finite for all $r,t$. Given that Bessel flow, ${\bf W}=(Y-\int_0^.\frac{\tau}{2R^r_s}ds,{\bf Z})$
solves an ODE of type
\begin{equation}\label{eq:forcedODE} 
\frac{d}{dt}{\bf W}_t={\bf e}_t({\bf W}_t)
\end{equation}
with $w\mapsto {\bf e}_t(w)$ smooth (locally uniformly in $t,w$) and this ODE flow is thus well-defined for a small positive time (depending on the realization of the driving BM). Indeed, let $B_0=B_0({\bf 0},r_0)$ be a neighborbood of ${\bf 0}$ where $b,{\bf c}$ are bounded by a constant $k$; and let
$$\sigma=\inf\{s\geq 0: V^{\bf v}_s\notin B_0\}$$

Let $\tau_r^{r'}=\inf\{s\geq 0:R^r_s\geq r'\}$ which is a.s. finite (we are interested in the case where $0<r\ll r'\ll 1$). We have $\tau_r^{r'}\leq\tau_0^{r'}$; and by scaling and the Markov property we have the exponential tail estimate
$$\P(\tau_0^{r'}\geq t)\leq c\exp(-c^{-1}t/(r')^2)$$
for some $c>0$.

Taking into account the defining relation $R^r_t-r=B_t+\frac{\delta-1}2\int_0^t\frac{ds}{R^r_s}$ and the standard Gaussian tail estimate $\P(\sup_{0\leq s\leq t} |B_s|\geq k)\leq ce^{-k^2/(2t)}$ for a BM started from $0$ (e.g. from the reflection principle), we see that for $\alpha$ large enough
\begin{align*}
\P(\int_0^{\tau_0^{r'}}\frac{ds}{R^0_s}\geq \alpha r')
&\leq \P(\tau_0^{r'}\geq \alpha(r')^2)+\P(\sup_{0\leq s\leq \alpha(r')^2} |B_s|\geq c^{-1}\alpha r')\\
&\leq c(e^{-\alpha/c}+e^{-\alpha^2/c\alpha})\leq ce^{-\alpha/c}
\end{align*}
for some $c>0$ (as usual, $c$ denotes everywhere is a constant independent of parameters such as $r'$, but its exact value may change from line to line). Remark that $\int_0^{\tau_r^{r'}}\frac{ds}{R^r_s}\leq \int_0^{\tau_0^{r'}}\frac{ds}{R^0_s}$. 

On the event $\{\int_0^{\tau_0^{r'}}\frac{ds}{R^0_s}\leq r_0/2\}$ and before $\tau_r^{r'}$ and the first exit of $B_0$ by $(V^{\bf v}_t)$, we have $|Y_t-Y^{\bf v}_0|\leq kt$, $|{\bf Z}^{\bf v}_t-{\bf Z}^{\bf v}_0|\leq kt$. Consequently if ${\bf v}\in B_0/2$, $(V^{\bf v}_t)$ stays in $B_0$ at least until $t=c^{-1}r_0$. %
Before exiting $B_0$ we have
$$|{\bf V}^{{\bf v}_2}_t-{\bf V}^{{\bf v}_1}_t|\leq |{\bf v}_2-{\bf v}_1|+k\int_0^t|{\bf V}^{{\bf v}_2}_s-{\bf V}^{{\bf v}_1}_s|ds$$
if ${\bf v}_1,{\bf v}_2\in B_0/2$. Thus (Gronwall inequality)
$$|{\bf V}^{{\bf v}_2}_t-{\bf V}^{{\bf v}_1}_t|\leq |{\bf v}_2-{\bf v}_1|e^{kt}$$
and moreover 
${\bf V}^{\bf 0}_t\leq c(|R_t|+|\int_0^t ds/R^0_s|+e^{kt})$.

In summary, on an event of probability $\geq 1-c^{-1}e^{-1/cr'}$, the flow started from $B_0/2\cap ((0,r']\times\R^{n+1})$ exits $B_0\cap((0,r']\times\R^{n+1})$ on $\{r'\}\times\R^{n+1}$ and is 2-Lipschitz (at a fixed time before exiting).%

We proceed with a study of the regularity of the flow. It is manifestly not smooth in $t$ and, as pointed out earlier, it is also not smooth in $r$ at $r=0$. However it is smooth in $y,{\bf z}$, as we shall now show. 

Given a fixed realization of the Bessel flow, ${\bf W}_t=(Y^{\bf v}_t-\frac\tau 2\int_0^t\frac{ds}{R^r_s},{\bf Z}^{\bf v}_t)_{t,{\bf v}}$ is a (time-dependent) ODE flow \eqref{eq:forcedODE}. Classically \cite{Arnold_ODE}, for short times
$$\frac{d}{dt}\partial_{w_i}{\bf W}^{\bf w}_t={\rm Jac}({\bf e}_t)({\bf W}_t^{\bf w})\partial_{w_i}{\bf W}_t^{\bf w}$$ 
viz. given the flow ${\bf W}$, the Jacobian flow ${\rm Jac}({\bf W}^{\bf w}_t)$ solves a time-dependent linear ODE; by iterating the same holds for higher-order derivatives.

Moreover, $R^r_t$ and consequently $\frac{\delta-1}2\int_0^t\frac{ds}{R^r_s}=R^r_t-B_t$ have a $C^1$ dependence in $r$. Consequently ${\bf e}_t$ and its derivatives w.r.t the $w$ variables is $C^1$ in $r$. Again by standard results on parametric ODEs (\cite{Arnold_ODE}), ${\bf W}^{\bf w}_t$ and its partial derivatives w.r.t. the $w$ variables are $C^1$ in $r$.

In the following lemma we collect what we will need on the regularity of the flow up to first exit of a thin slab; specifically, we consider the flow starting in $(B_0/2)\cap([0,r]\times\R^{n+1})$ and up to exit of a mesoscopic (in the $r$ direction) slab $B_0\cap([0,r^\alpha]\times\R^{n+1})$. With high probability (w.h.p.), this occurs in time $O(r^{2\alpha-\eps})$, during which $\int_0^.ds/R_s^0$ is (w.h.p.) staying small and consequently the full flow is also small. Moreover the $C^k$ norm of the flow (in the $y,z$ directions) is bounded. 

\begin{Lem}\label{Lem:flowslab}
Fix $k\geq 0$, $\alpha\in(0,1)$; set $B_0$ a small enough neighborhood of ${\bf v}_0={\bf 0}$, $K$ large enough and $r_0>0$ small enough. Then for $r\in(0,r_0)$ small enough, on an event $E=E(r,\alpha)$ of probability $\geq 1-ce^{-1/cr^\eps}$, the flow \eqref{eq:SDEflow} has the following properties:
\begin{enumerate}
\item The flow started from $(B_0/2)\cap( [0,r]\times \R^{n+1})$ is defined up to time $t_0=r^{2\alpha-\eps}$; stays in $B_0\cap ((0,r_0]\times \R^{n+1})$ for $t\in (0,t_0]$; exits $B_0\cap ([0,r^\alpha]\times \R^{n+1})$ on $\{r^\alpha\}\times \R^{n+1}$ and before $t_0$ (first exit strictly after $t=0$). 
\item The flow has a continuous extension to $[0, t_0]\times \left((B_0/2)\cap([0,r]\times\R^{n+1})\right)$. 
\item For fixed $t=t(\omega)\leq t_0$, the flow is $C^1$ in $r$ with $C^1$ norm $\leq K$
\item For fixed $t=t(\omega)\leq t_0$, the flow is $C^k$ in ${\bf w}$ with $C^k$ norm $\leq K$
\end{enumerate} 
\end{Lem}

\subsubsection{Regularity}

Here we will conclude the proof of Lemma \ref{Lem:reg}. The argument goes roughly as follows. From the continuity of the flow and a Feynman-Kac representation we obtain a continuous extension up to the boundary of a bounded $f$ s.t. ${\mc L}f=0$ (Lemma \ref{Lem:cont}). In the absence of potential and for globally bounded coefficients, we could represent $f$ at $r=0$ as a weighted average (over realizations of the flow) of $f$ at some $r=r_0>0$; the smoothness of the flow in the $y,z$ directions then gives smoothness of $f$ in these directions up to the boundary. Given this we are left with one variable, $r$, which can be dealt with in elementary fashion (Lemma \ref{Lem:paramODE}). In order to justify exchanging expectations (under the flow) and differentiation, we will need some {\em a priori} estimates, given by the mildness condition.

We denote ${\bf v}=(r,y,z_1,\dots,z_n)$. Recall ${\mc L}$ from \eqref{eq:diffopL}, and ${\mc L}={\mc L}_0+({\rm potential})$.

\begin{Lem}\label{Lem:cont}
Let $\eps>0$ be small enough and $f:(0,\ell)\times B_0\rightarrow \R$ be smooth and such that $f({\bf v})=O(r^{2-\delta+\eps})$ and 
$({\mc L}_0f)({\bf v})=O(r^{-2+\eps})$. Then $f$ extends continuously to a neighborhood of ${\bf 0}$ in $[0,\ell)\times B_0$.
\end{Lem}
The statement is essentially sharp, as may be seen by considering $f(r)=r^{2-\delta}$ and $f(r)=\log(r)$ when ${\mc L}_0=\frac 12\partial_{rr}+\frac{\delta-1}{2r}\partial_r$ is the Bessel generator.

\begin{proof}
Throughout $c>0$ is a large constant and $\eps>0$ is a small constant whose values may change from line to line.

By Dynkin's formula we have that 
$$M^f_t({\bf v})=f({\bf V}^{\bf v}_t)-\int_0^t g({\bf V}^{\bf v}_s)ds$$
is a local martingale, where $g={\mc L}_0f$. %
By a standard scale function argument for one-dimensional diffusions (e.g. \cite{RY}), we have 
$$\P(R^r{\rm\ hits\ }r''{\rm\ before\ }r')=\frac{r^{2-\delta}-(r')^{2-\delta}}{(r'')^{2-\delta}-(r')^{2-\delta}}\sim (r''/r)^{\delta-2}$$
for $0<r''<r<r'$. Consequently (recall that $\delta>2$),
\begin{equation}\label{eq:besmom}
\E(\max_{s\leq\tau_r^{r'}} R_s^{2-\delta+\eps})=(\delta-2-\eps)\int_{r^{-1}}^\infty\P(R^r{\rm\ hits\ }s^{-1}{\rm\ before\ }r')s^{\delta-3-\eps}ds<\infty
\end{equation}
Notice also that $|\int_0^{t\wedge\tau_r^{r'}} g(R^r_s,{\bf V}^{\bf v}_s)ds|\leq c\int_0^{\tau_r^{r'}}R_s^{-2+\eps}ds$, which is integrable. Indeed, 
recall that 
$$\E((R^0_t)^{-\eta})=c_\eta t^{-\eta/2}$$ 
by scaling, with $c_\eta<\infty$ if $\eta<\delta$ (e.g. from the explicit Bessel semigroup, see \cite{RY} XI.1). 
Then $\E(\int_0^1\frac{ds}{(R^0_s)^{2-\eps}})<\infty$; and by monotonicity of the flow $\E(\int_{n}^{n+1}\frac{ds}{(R_s^0)^{2-\eps}}|\tau_r^{r'}\geq n)$ is bounded; and $\tau_{r}^{r'}$ has exponential tails, which gives $\E\int_0^{\tau_r^{r'}}R_s^{-2+\eps}ds<\infty$.

 Consequently the stopped process $t\mapsto M^f_{t\wedge\tau_r^{r'}}$ is dominated by a multiple of the integrable variable
$$\max_{s\leq\tau_r^{r'}} R_s^{2-\delta+\eps}+\int_0^{\tau_r^{r'}}R_s^{\eps-2}ds$$
and is thus uniformly integrable and a (true) martingale. Then we have the representation
\begin{equation}\label{eq:dynkin}
f({\bf v})=\E\left(f({\bf V}^{\bf v}_{\sigma})-\int_0^{\sigma} g({\bf V}^{\bf v}_s)ds\right)
\end{equation}
for any stopping time $\sigma$ anterior to the first exit of $(0,\ell/2)\times U$, $U$ a small enough neighborhood of $(y_0,{\bf z}_0)$.%

We then argue that $f$ is bounded near ${\bf 0}$. To see this, 
notice that by the Markov property, transience, scaling, and the moment estimate \eqref{eq:besmom}, it follows that
$$\E(\sup_{s\geq t}(R^0_t)^{-\eta})=c'_\eta t^{-\eta/2}$$
with $c'_\eta<\infty$ for $\eta<\delta-2$.

Then we notice that the flow is (roughly speaking) exponentially unlikely to travel at macroscopic distance in short time. Specifically,
$$\P(\sup_{0\leq s\leq t}(|B_s|+R^0_s)\leq t^{\eps})\geq 1-ce^{-c^{-1}t^{2\eps-1}}$$
and on this event $\sup_{0\leq s\leq t}|R^r_s-r|\leq ct^{\eps}$, $\sup_{0\leq s\leq t}|Y_s-Y_0|\leq ct^{\eps}$ and 
$\sup_{0\leq s\leq t}|{\bf Z}_s-{\bf Z}_0|\leq ct$. Let $\sigma$ be the time of first exit of $(0,\ell/2)\times U$ by 
the flow started anywhere in $(0,\ell/4)\times U'$, where $U'$ a neighborhood of $(y_0,{\bf z}_0)$ strongly included in $U$.
Then we have the following (rather crude but sufficient for our purposes) estimate:
$$\P(\sigma\leq t)=O(e^{-t^{2\eps-1}})$$
Going back to \eqref{eq:dynkin} for the stopping time $\sigma\wedge 1$, we have 
$${\bf v}\mapsto\E\left(\int_0^{\sigma\wedge 1} g({\bf V}^{\bf v}_s)ds\right)$$
is uniformly bounded in $(0,\ell/4)\times U'$ by the earlier domination argument; and 
\begin{align*}
|f({\bf V}^{\bf v}_{\sigma\wedge 1})|
&\leq c (R^0_{\sigma\wedge 1})^{-\eta}\\
&\leq c\sum_{k\geq 0}\ind_{\sigma\in[2^{-k-1},2^{-k}]}\sup_{2^{-k-1}\leq s\leq 2^{-k}} (R^0_s)^{-\eta}
\end{align*}
with $\eta=\delta-2-\eps$. For $p>1$, $p$ close enough to 1, we have 
$$\|\sup_{2^{-k-1}\leq s\leq 2^{-k}} (R^0_s)^{-\eta}\|_{L^p}=O(2^{k\eta/2})$$
and for $q$ the H\"older conjugate of $p$, 
$$\|\ind_{\sigma\in[2^{-k-1},2^{-k}]}\|_{L^q}=O(e^{-q^{-1}2^{k(1-2\eps)}})$$
which by \eqref{eq:dynkin} gives a uniform bound on $f$ in $(0,\ell/4)\times U'$.

Finally we check continuity at a point ${\bf v}_1$ on $\{0\}\times U'$. We consider again \eqref{eq:dynkin} at the stopping time $\tau_{0}^{r'}\wedge\sigma$, $\sigma$ the time of first exit of $(0,\ell/8)\times U''$ (here $r'\ll 1$ is to be specified, $U''\subset\subset U'$). Take $\eps_0>0$ arbitrarily small. Set $r'>0$ small enough so that $\E(\int_0^{\tau_0^{r'}}\frac{ds}{(R^0_s)^{2-\eps}})\leq c^{-1}\eps_0$; and $\P(\tau_0^{r'}\geq \sigma)\leq c^{-1}\eps_0$ for ${\bf v}$ close to ${\bf v}_1$. Then
$$|f({\bf v})-\E(f({\bf V}^{\bf v}_{\tau_0^{r'}})\phi({\bf V}^{\bf v}_{\tau_0^{r'}}))|\leq\eps_0$$
for ${\bf v}$ close enough to ${\bf v}_1$ by \eqref{eq:dynkin}, where $0\leq\phi\leq 1$ is continuous, equal to 1 in $(0,\ell/8)\times U''$ and to $0$ outside of $(\ell/4)\times U'$. Moreover, by dominated convergence (since we now know that $f$ is bounded near ${\bf v}_1$, and the flow is continuous)
$${\bf v}\mapsto\E(f({\bf V}^{\bf v}_{\tau_0^{r'}})\phi({\bf V}^{\bf v}_{\tau_0^{r'}}))$$
is continuous at ${\bf v}_1$, from which we conclude that $f$ itself is continuous at ${\bf v}_1$. (We can compactify the state space of the flow by adding a cemetery point $\partial$ and setting $(f\phi)(\partial)=0$, which deals with possible explosions of the flow prior to $\tau_0^{r'}$).  
\end{proof}

In the next lemma we exploit the fact that the flow is smooth in the $w$ directions (viz. $y,z_1,\dots,z_n$) in order to improve {\em a priori} mildness estimates.

\begin{Lem}\label{Lem:gradgrowth}
Let $U_0=(0,\ell)\times B_0$, $B_0$ a small bounded neighborhood of $(0,{\bf 0})$. Assume that $f: U_0\rightarrow\R$ is smooth and bounded, ${\mc L}f=0$ in $U_0$. Assume that the partial derivatives of $f$ w.r.t the $w$ variables up to order $k+1$ are $O(r^{-M})$ near the boundary for some $M>0$. Then the partial derivatives of $f$ w.r.t the $w$ variables up to order $k$ are $O(r^{-\eps})$ near the boundary for all $\eps>0$.
\end{Lem}
\begin{proof}
We have ${\mc L}={\mc L}_0+\frac dr$, $d$ smooth up to the boundary. Let
$$M^f_t=\exp\left(\int_0^t \frac{d({\bf V}^{\bf v}_t)}{R^r_s}ds\right)f({\bf V}^{\bf v}_t)$$
and $\sigma$ the first exit of $U_0$. Then (Feynman-Kac) $M^f_{\sigma\wedge .}$ is a local martingale. Moreover $\int_0^\sigma(R^r_s)^{-1}ds$ is bounded (consider e.g. the $Y_\sigma$ component, which is itself bounded). Then $M^f_{\sigma\wedge .}$ is a bounded martingale, which justifies the representation
\begin{equation}\label{eq:FK}
f({\bf v})=\E\left(\exp\left(\int_0^\tau d({\bf V}^{\bf v}_s)\frac{ds}{R^r_s}\right)f({\bf V}^{\bf v}_\tau)\right)
\end{equation}
for $\tau$ any stopping time anterior to $\sigma$. 

Let $\eta$ be a small parameter and consider the discrete derivative $(D_ig)({\bf v})=g({\bf v}+\eta e_i)-g({\bf v})$, where $e_i$ is the basis vector corresponding to the coordinate $w_i$. We have
$$D_{i_1}\dots D_{i_k}g=\eta^k\partial^k_{w_{i_1}\dots w_{i_k}}g+O(\eta^{k+1}\|g\|_{C^{k+1}})$$
We reason on an event $E$ as in Lemma \ref{Lem:flowslab}. On that event, $\int_0^{\tau_r^{r^\alpha}}ds/R^r_s$ is bounded; the flow and its derivatives in the $w$ directions up to order $k$ are bounded; and the derivatives of $d$ to any order are bounded. It follows that on $E$
$$ D_{i_1}\dots D_{i_k}\left(\exp\left(\int_0^{\tau_r^{r^\alpha}}d({\bf V}^{\bf v}_s)\frac{ds}{R^r_s}\right) f(V^{\bf v}_{\tau_r^{r^\alpha}})\right)=O(\eta^kr^{-\alpha M})$$
Then
$$\partial^k_{w_{i_1}\dots w_{i_k}}f({\bf v})=O(\eta^{-k}e^{-r^{-\eps}}+\eta r^{-M}+r^{-\alpha M})$$
By taking $\eta=r^M$ and $\alpha$ sufficiently small, the result follows.
\end{proof}

The following simple ODE lemma will allow us to translate regularity of $f$ in the $w$ directions into regularity in the $r$ direction. Recall that $\delta>2$. 

\begin{Lem}\label{Lem:paramODE}
Let $U$ be a neighborhood of ${\bf 0}$ in $\R^n$ and $\ell>0$; a point in $[0,\ell)\times U$ is denoted by $(r,s_1,\dots,s_n)$. Let $f:(0,\ell)\times U\rightarrow\R$ be bounded and smooth and $g:[0,\ell)\times U$ be such that $g,\partial_rg,\dots,\partial^k_rg$ are continuous in $(r,s_1,\dots)$ and smooth in $(s_1,\dots)$. Assume furthermore that
$$\frac 12\partial_{rr}f+\frac{\delta-1}{2r}\partial_rf=\frac gr$$
in $(0,\ell)\times U$. Then $\partial_r f,\dots,\partial^{k+1}_rf$ extend continuously to $[0,\ell)\times U$ and are smooth in $(s_1,\dots,s_n)$
\end{Lem}
\begin{proof}
We write 
$$r^{1-\delta}\partial_r\left(r^{\delta-1}\partial_rf\right)=2\frac{g}{r}$$
and thus
$$\partial_r f=r^{1-\delta}\left(h(s_1,\dots,s_n)+2\int_0^rg(t,s_1,\dots)t^{\delta-2}dt\right)$$
Since $f$ is bounded, $h(s_1,\dots,s_n)=0$. Then
$$\partial_rf=2\int_0^1g(ur,s_1,\dots,s_n)u^{\delta-2}du$$
from which the assertion follows by differentiating under the integral.
\end{proof}

Finally we obtain Lemma \ref{Lem:reg}, stated here for the operator ${\mc L}$ \eqref{eq:diffopL}, which is conjugate to ${\mc M}$ \eqref{eq:diffop0}.

\begin{Lem}
If ${\mc L}f=0$ in $U_0=(0,\ell)\times B_0$ and $f$ is mild and $f=O(r^{2-\delta+\eps})$ for some $\eps>0$, then $f$ and all its derivatives extend continuously to $[0,\ell/2]\times B_0/2$. 
\end{Lem}
\begin{proof}
From Lemma \ref{Lem:cont}, $f$ extends continuously to the boundary. From the mildness condition, $\nabla^kf=O(r^{-M_k})$ for some $M_k$. Then from Lemma \ref{Lem:gradgrowth} it follows that the derivatives (of any order) of $f$ in the $w$ directions are $O(r^{-\eps})$ for $\eps$ arbitrarily small.

We have $\frac{\partial}{\partial w_i}r=0$; and thus
$${\mc L}(\partial^k_{w_{i_1}\dots w_{i_k}}f)=[{\mc L},\partial^k_{w_{i_1}\dots w_{i_k}}]f=\sum_{j_1,\dots,j_k}\alpha_{w_{j_1}\dots w_{j_k}}\partial^k_{w_{j_1}\dots w_{j_k}}f
+r^{-1}\sum_{j_1,\dots,j_{k-1}}\beta_{j_1\dots j_{k-1}}\partial^{k-1}_{w_{j_1}\dots w_{j_{k-1}}}f
$$
where the $\alpha$'s and $\beta$'s are smooth up to the boundary. Since $\partial^k_{w_{i_1}\dots w_{i_k}}f=O(r^{-\eps})$ and ${\mc L}\partial^k_{w_{i_1}\dots w_{i_k}}f=O(r^{-1-\eps})$, by another application
 of Lemma \ref{Lem:cont} we see that $\partial^k_{w_{i_1}\dots w_{i_k}}f$ extends continuously to the boundary.

We consider now ${\mc L}f=0$ as a parametric ODE:
$$\frac 12\partial_{rr}f(r,w_1,\dots)+\frac{\delta-1}{2r}\partial_rf(r,w_1,\dots)=\frac1 r\left(-df(r,\dots)+\sum_i\alpha_i\partial_{w_i}f(r,\dots)\right)$$
where $d$ and the $\alpha_i$'s are smooth up to the boundary. The RHS is continuous in $r$ and smooth in the $w_i$'s up to the boundary. By Lemma \ref{Lem:paramODE}, $\partial_rf$ and its derivatives in the $w$ directions are continuous up to the boundary. By repeated applications of Lemma \ref{Lem:paramODE}, the same holds for $\partial_r^kf$.

\end{proof}

\section{Fusion}

In this section we consider the framework for (pairs of commuting) Virasoro representations described in details e.g. in Section 4.3 of \cite{Dub_Virloc}. Specifically $(\Sigma,X,Y,\dots)$ is a bordered Teichm\"uller surface with two marked points $X,Y$ on the same boundary component; we are interested in the regime where $Y\rightarrow X^+$ (the boundary components are oriented, with the surface lying to their left-hand side). Additional points may be marked in the bulk or on the boundary (but not between $X,Y$). Then ${\mc T}_2$ denotes the Teichm\"uller space of such marked surfaces. We consider an extended Teichm\"uller space $\hat{\mc T}_2$ corresponding to surfaces with an additional marking consisting of $\tilde z$, $\tilde w$, formal local coordinates at $X,Y$. It can be realized as the projective limit of (finite-dimensional, smooth) Teichm\"uller surfaces. 

Similarly, we consider the Teichm\"uller space ${\mc T}_1$ of surfaces of type $(\Sigma,X,\dots)$ ($Y$ is omitted); and the space ${\mc T}_0$ of surfaces of type $(\Sigma,\dots)$ ($X,Y$ are omitted), so that we have natural projections ${\mc T}_2\rightarrow{\mc T}_1\rightarrow{\mc T}_0$. Likewise, $\hat{\mc T}_1$ is the extended Teichm\"uller space keeping track of a formal coordinate at $X$.

Over $\hat{\mc T}_2$, we consider the determinant bundle ${\mc L}$, whose sections can be identified with functional of a Riemannian metric satisfying a certain conformal anomaly formula. Then there are two commuting representations $(L^X_n)_{n\in\Z}$, $(L_n^Y)_{n\in\Z}$ of the Virasoro algebra (with central charge $c$) operating on smooth sections of ${\mc L}^{\otimes c}$. They correspond to infinitesimal deformations of the surface at $X,Y$ respectively. Over $\hat{\mc T}_1$, we have a single Virasoro representation corresponding to deformations at $X$.

We consider a section ${\mc Z}$ which (at least in some open set) is highest-weight (viz. depends on $\tilde z$ - resp. $\tilde w$ - as a $h_X$-form - resp. as a $h_Y$-form); and is a null vector for both representations, in the sense that
\begin{align*}
\Delta^X_{r,s}{\mc Z}&=0\\
\Delta^Y_{r',s'}{\mc Z}&=0
\end{align*}
where $\Delta^X_{r,s}$ (resp. $\Delta^Y_{r,s}$) is the image of the singular vector $\Delta_{r,s}\in{\mc U}(\Vir)$ in the representation by deformation at $X$ (resp. at $Y$). Let $z$ be a (genuine) local coordinate at $X$ and assume that for $Y$ close to $X$,
$${\mc Z}(\Sigma,X,Y,z,z-z(Y),\dots)=(z(Y)-z(X))^{h-h_X-h_Y}\left({\mc Z}_0(\Sigma,X,z,\dots)+o(1)\right)$$
with ${\mc Z}_0$ finite and non-vanishing. Then ${\mc Z}_0$ (defined on the Teichm\"uller space $\hat{\mc T}_1$ of surfaces of type $(\Sigma,X,\tilde z,\dots)$) is itself highest-weight and one expects that it satisfies a null vector equation. We will show that this is indeed the case if $(r',s')=(2,1)$ (or dually $(1,2)$). The problem is essentially to go from a pair of commuting representations on sections on $\hat{\mc T}_2$ to a single representation on sections on $\hat{\mc T}_1$.

There are two main steps in the proof to the main result, Theorem \ref{Thm:fus}. First we will justify the existence of an asymptotic expansion
\begin{equation}\label{eq:asympexp}
{\mc Z}(\Sigma,X,Y,z,z-z(Y),\dots)=(z(Y)-z(X))^{h-h_X-h_Y}\left(\sum_{k=0}^n{\mc Z}_k(\Sigma,X,z,\dots)(z(Y)-z(X))^k+O((z(Y)-z(X))^{n+1})\right)
\end{equation}
for arbitrarily large $n$. This is an analytic argument based on the condition $\Delta^Y_{2,1}{\mc Z}=0$ and Lemma \ref{Lem:mainreg}. Then we will deduce from the null vector equations relations between the terms ${\mc Z}_0,\dots,{\mc Z}_k,\dots$ of the asymptotic equation and will show they imply a null vector equation for the leading term ${\mc Z}_0$; that is the algebraic step, based on Lemma \ref{Lem:algfus}.

At this stage the main difficulty is to show that we are indeed in the situations abstracted out in Lemmas \ref{Lem:algfus} and \ref{Lem:mainreg} respectively; this is the role of Lemmas \ref{Lem:loccomp} to \ref{Lem:defW}. In order to do this, we need to write both null vector equations in the same set of coordinates. The basic issue here is that while so far we have been using formal local coordinates (with possibly zero as radius of convergence), we will now need, as an intermediate step, to consider families of (genuine) local coordinates defined on a neighborhood containing both $X$ and $Y$. 

We may assume that the additional marking includes spectator points $Z,Z'$ on the same boundary component as $X$ and $Y$ (in cyclic order $Z',Z,X,Y$). We can choose a meromorphic 1-form $\sigma=\sigma(\Sigma,Z')$ with a pole at $Z'$, real and positive along the boundary near $X$, and depending smoothly on the conformal class of $\Sigma$ (e.g. by taking derivatives of the Green kernel). %
This gives a function $z_\Sigma=\int_Z^.\sigma$ (integrating along the boundary component) defined near $X$, analytic in the position and smooth in $(\Sigma,Z,Z',\dots)$.

Let $N\in\N$ be a large integer. If $\tilde z$ is an $N$-jet of local coordinate at $X$, there is a unique polynomial $P=P(\tilde z,z_\Sigma)$ of degree $N$ s.t. $\tilde z$ is the $N$-jet of $P(z_\Sigma-z_\Sigma(X))$. The mapping 
$$f:(\Sigma,X,Y,\tilde z,\dots)\longmapsto P(z_\Sigma(Y)-z_\Sigma(X))$$
defines a smooth function on Teichm\"uller space of surfaces with a marked $N$-jet, and thus lifts to a smooth function on the extended Teichm\"uller space $\hat{\mc T}_2$. We will also consider 
$$g:(\Sigma,X,Y,\tilde z,\tilde w)\longmapsto \frac{dP(z_\Sigma-z_\Sigma(X))}{dw}(Y),$$
which also depends on a 1-jet of local coordinate at $Y$. Remark that changing $z_\Sigma$ to another function $z'_\Sigma$ changes $f$ to $f'$ with $f-f'=O(|z(Y)-z(X)|^{N+1})$.

\begin{Lem}\label{Lem:loccomp}
We have
\begin{align*}
(\ell_n^Xf)(\Sigma,X,Y,\tilde z,\dots)&=-f^{n+1}+o(|z(Y)-z(X)|^{N+n+1})\\
(\ell_n^Xg)(\Sigma,X,Y,\tilde z,\tilde w,\dots)&=(n+1)f^ng+o(|z(Y)-z(X)|^{N+n})
\end{align*}
for $n\in\Z$.
\end{Lem}
\begin{proof}
By construction of the $\ell_n$'s \cite{Dub_Virloc}, there is a one-parameter family of surfaces $(\Sigma_t,X_t,Y,z_t)$ identified outside of a neighborhood of $X$ s.t. $z_t$ is a local coordinate at $X_t$, $z_t=z-tz^{n+1}+o(t)$ in a semi-annulus centered at $X$, and $(\ell_n f)(\Sigma,\dots)={\frac{d}{dt}}_{|t=0}f(\Sigma_t,\dots)$. We can choose this semi-annulus to contain $Y$ and $z_0=P(\tilde z,z_\Sigma)$ (for $Y$ close enough to $X$ given $\tilde z$). Then
$${\frac{d}{dt}}_{|t=0}(z_t(Y)-z_t(X_t))=-f^{n+1}$$
Moreover $(\Sigma_t,X_t,\tilde z)$ is identified with $(\Sigma,X,\tilde z)$ near $X$ and consequently the $N$-jets of $z_t$ and $P(\tilde z,z_{\Sigma_t})$ agree at $X$, so that $f(\Sigma_t)=z_t(Y_t)+O(z_t(Y_t)^{N+1})$ (where the $O$ is differentiable in $t$). This gives the expression for $\ell_n^Xf$. The argument for $\ell_n^Xg$ is similar.
\end{proof}

Given a smooth section ${\mc Z}$ on $\hat{\mc T}_2$ with an expansion \eqref{eq:asympexp}, we can consider the ``descendants" $({\mc Z}_k)_{k\geq 0}$ as smooth sections on $\hat{\mc T}_1$. We now want to express the action of $L_m^X$ on ${\mc Z}$ in terms of the action of $L_m$ on the descendants.

\begin{Lem}\label{Lem:defX}
Assume that ${\mc Z}$ is $h_Y$-highest-weight at $Y$ and has an asymptotic expansion of type:
$${\mc Z}(\Sigma,X,Y,z,z-z(Y),\dots)=\eps^{\alpha}\left(\sum_{k=0}^n{\mc Z}_k(\Sigma,X,z,\dots)\eps^k+\eps^{n+1}{\mc Z}_n^r(\Sigma,X,Y,z,\dots)\right)
$$
with ${\mc Z}_n^r$ smooth for each $n$ and $\eps=z(Y)-z(X)$. Then $L_m^X{\mc Z}$ has a similar expansion
$$(L^X_m{\mc Z})(\Sigma,X,Y,z,z-z(Y),\dots)=\eps^{\alpha}\left(\sum_{k=m}^n(L_m^X{\mc Z})_k(\Sigma,X,z,\dots)\eps^k+\eps^{n+1}(L_m^X{\mc Z})_n^r(\Sigma,X,Y,z,\dots)\right)
$$
and
$$(L_m^X{\mc Z})_k=L_m{\mc Z}_k-(\alpha+k-m){\mc Z}_{k-m}$$
\end{Lem}
\begin{proof}
We may write
$${\mc Z}(\Sigma,X,Y,z,\tilde w,\dots)=f^\alpha g^{-h_Y}\left(\sum_{k=0}^n{\mc Z}_k(\Sigma,X,z,\dots)f^k+f^{n+1}{\mc Z}_n^r(\Sigma,X,Y,z,\dots)\right)
$$
and apply Lemma \ref{Lem:loccomp} for $N$ large enough to get
$$(L_m^X{\mc Z})(\Sigma,X,Y,z,\tilde w,\dots)=f^\alpha g^{-h_Y}\left(\sum_{k=0}^n(f^k(L_m{\mc Z}_k-f^{m}((k+m+\alpha)+h_Y(m+1)){\mc Z}_k)+O(f^{N+m})\right)
$$
which gives the formula for $L_m^X{\mc Z}_k$.
\end{proof}

We turn to the deformation at $Y$. This is complicated by the fact that the local coordinate is centered at $X$; but we need only evaluate $L_{-1}^Y$ and $L_{-2}^Y$.

\begin{Lem}\label{Lem:defY}
Assume that ${\mc Z}$ is $h_X$-highest-weight at $X$ and has an asymptotic expansion as in Lemma \ref{Lem:defX}. 
Then $L_m^Y{\mc Z}$ has a similar expansion
$$(L^Y_m{\mc Z})(\Sigma,X,Y,z,z-z(Y),\dots)=\eps^{\alpha}\left(\sum_{k=m}^n(L_m^X{\mc Z})_k(\Sigma,X,z,\dots)\eps^k+\eps^{n+1}(L_m^X{\mc Z})_n^r(\Sigma,X,Y,z,\dots)\right)
$$
 for $m=-1,-2$ and
\begin{align*}
(L_{-1}^Y{\mc Z})_k&=(\alpha+k){\mc Z}_k\\
(L_{-2}^Y{\mc Z})_k&=\ell_{-1}{\mc Z}_{k+1}-(\alpha+k+2-h_X){\mc Z}_{k+2}+\sum_{m\geq 0}\ell_{-2-m}{\mc Z}_{k-m}
\end{align*}
\end{Lem}
\begin{proof}
For $L_{-1}^Y$, we have the representation
$$L_{-1}^Y{\mc Z}(\Sigma,X,Y,z,z-z(Y),\dots)={\frac{d}{dt}}_{|t=0}{\mc Z}(\Sigma,X,Y_t,z,z-z(Y_t),\dots)$$
where $Y_t$ is the point s.t. $z(Y_t)=z(Y)+t$, from which we read off immediately
$$(L_{-1}^Y{\mc Z})_k=(\alpha+k){\mc Z}_k$$
To write the action of $L_{-2}^Y$ on the ${\mc Z}_k$'s, it is rather convenient to go back to its construction by Kodaira-Spencer deformation. The operator $\ell_{-2}^Y$ (evaluated at the local coordinate $z-z(Y)$) corresponds to the deformation given by the vector field $-(z-z(Y))^{-1}\partial_z$ on a small semi-annulus $A_Y$ around $Y$ (not containing $X$). In order to express it in terms of the $(\ell_n^X)$, defined relatively to a deformation around $X$, we can consider an equivalent \v Cech cocycle. Let $A_X$ be a semi-annulus around $X$ (disconnecting $X$ from $Y$ and the other markings) and $A$ a semi-annulus around $X,Y$ (disconnecting $X,Y$ from the other markings); we choose these three semi-annuli to be disjoint. 

The deformation given by $\ell_{-2}^Y$ is thus equivalent to sum of the deformation given by $(z-z(Y))^{-1}\partial_z$ in $A_X$ and the one given $(z-z(Y))^{-1}\partial_z$ in $A$. For the first one, we expand near $X$ to get
$$\frac{1}{z-z(Y)}\partial_z=\left(\frac{1}{z(X)-z(Y)}-\frac{z-z(X)}{(z(X)-z(Y))^2}+O((z-z(X))^2)\right)\partial_z$$
viz. the deformation is given by
$$-\eps^{-1}\ell_{-1}^X-\eps^{-2}\ell_0^X$$
modulo a term vanishing on highest-weight vectors (at $X$), and we have already analyzed the action of $\ell_{-1}^X$ on terms of the expansion. For the second deformation (in $A$, where $z-z(X)$ is large compared to $z(X)-z(Y)$), we expand
$$\frac{1}{z-z(Y)}\partial_z=\left(\frac{1}{z-z(X)}+\frac{\eps}{(z-z(X))^2}+\cdots+\frac{\eps^n}{(z-z(X))^{n+1}}+O(\eps^{n+1})\right)\partial_z
$$
Since $X,Y$ are inside $A$, $z(X)-z(Y)$ is constant under this second deformation. We thus obtain the expression:
$$(\ell_{-2}^Yf)_k=\ell_{-1}f_{k+1}-(\alpha+k+2-h_X)f_{k+2}+\sum_{m\geq 0}\ell_{-2-m}f_{k-m}$$
if $f$ is $h_X$-highest-weight at $X$ with an expansion 
$$f(\Sigma,X,Y,z-z(X),z-z(Y),\dots)=\eps^{\alpha}\sum_{k\geq 0}\eps^k f_k(\Sigma,X,z-z(X),\dots)$$
This gives the case $c=0$. In the general case, we write
$$L_{-2}^Y(fs_\zeta^c)=\left(\ell_{-2}^Yf+\frac c{12}S(\Sigma,Y,z-z(Y),\dots)f\right)s_\zeta^c$$
(where $s_\zeta$ denotes the reference section of the determinant bundle, see \cite{Dub_Virloc}), and expand
$$S(\Sigma,Y,z-z(Y),\dots)=\sum_{k\geq 0}\frac{\eps^k}{k!}(\ell_{-1}^X)^kS(\Sigma,X,z-z(X),\dots)$$
Comparing with the expression (for $m\geq 0$)
$$L_{-2-m}(fs_\zeta^c)=\left(\ell_{-2-m}f+\frac{c}{12}\frac{(\ell_{-1})^mS}{m!}f\right)s_\zeta^c$$
concludes.
\end{proof}

We will also need the (comparatively trivial) case where the deformation occurs at a boundary point $W$ distinct from $X,Y$; this will be needed in order to iterate the fusion process.

\begin{Lem}\label{Lem:defW}
Assume that ${\mc Z}$ is $h_X$- (resp. $h_Y$-) highest-weight at $X$ (resp. $Y$) and has an asymptotic expansion as in Lemma \ref{Lem:defX}. 
Then $L_m^W{\mc Z}$ has a similar expansion
$$(L^W_m{\mc Z})(\Sigma,X,Y,z,z-z(Y),\dots)=\eps^{\alpha}\left(\sum_{k=m}^n(L_m^W{\mc Z})_k(\Sigma,X,z,\dots)\eps^k+\eps^{n+1}(L_m^W{\mc Z})_n^r(\Sigma,X,Y,z,\dots)\right)
$$
and ${\mc Z}_0$ is $(h_X+h_Y+\alpha)$-highest-weight with 
$$(L_m^W{\mc Z})_0=L_m^W{\mc Z}_0$$
for $m\in\Z$.
\end{Lem}
\begin{proof}
The goal here is to relate the action of $L_m^W$ on sections on $\hat{\mc T}_2$ and $\hat{\mc T}_1$ (the formal coordinate is now at $W$). We use the same choice of coordinates as before Lemma \ref{Lem:loccomp} and write
$$f(x,y,w_1,\dots)={\mc Z}(\Sigma,X,Y,\dots,z,z-z(Y),\dots)/s_\zeta^c$$
In such a trivialization, by construction 
\begin{align*}
(L_m^W{\mc Z})(\Sigma,X,Y,W,\dots)/s&=\sum_i g_i(w_1,\dots)\partial_{w_i}f+g(x,w_1,\dots)\partial_xf+g(y,w_1,\dots)\partial_yf\\
&+\left(cU(w_1,\dots)+h_Xk(x,w_1,\dots)+h_Yk(y,w_1,\dots)\right)f
\end{align*}
where $w_1,\dots$ are coordinates on the extended Teichm\"uller space of surfaces of type $(\Sigma,W,\dots)$ (including a formal local coordinate at $W$ and all the markings except $X,Y$), and the $g,k,g_i$'s are smooth coefficients. 

More explicitly, if $(\Sigma_t,W_t,\dots)_{t\geq 0}$ is a one-parameter family of surfaces representing the deformation $\ell_m^W$ (all identified together away from $W$) and $z_t=z_{\Sigma_t}$ is the corresponding local coordinate near $X$, we have 
$$k(x,w_1,\dots)={\frac{d}{dt}}_{|t=0}\frac{dz_t}{dz}(X)$$
Besides, by our choice of coordinates $y_t-x_t=\int_X^Ydz_t$, from which we get $\partial_xg(x,w_1,\dots)=k(x,w_1,\dots)$.

Then we write the asymptotic expansion 
$$f(x,y,w_1,\dots)=(y-x)^\alpha \left(\sum_{k=0}^Nf_k(x,w_1,\dots)(y-x)^k+o((y-x)^{\alpha+N})\right)$$ and it follows that
\begin{align*}
(L_m^W{\mc Z})(\Sigma,X,Y,W,\dots)/s&=(y-x)^{\alpha}\sum_{k=0}^Nu_k(x,w_1,\dots)+o((y-x)^{\alpha+N})
\end{align*}
with
\begin{align*}
u_0(x,w_1,\dots)&=\sum_i g_i(w_1,\dots)\partial_{w_i}u_0+g(x,w_1,\dots)\partial_xu_0\\
&+\left(cU(w_1,\dots)+(h_X+h_Y+\alpha)k(x,w_1,\dots)\right)u_0
\end{align*}
and the following terms $u_1,\dots$ also involve derivatives of $k$ w.r.t. $x$. This justifies the expansion for $L_m^W{\mc Z}$ and the expression
$$(L_m^W{\mc Z})_0=L_m^W{\mc Z}_0$$
\end{proof}

We can now phrase our main result, obtained by combining the algebraic elimination argument of Lemma \ref{Lem:algfus} with the regularity estimates of Lemma \ref{Lem:mainreg}. Lemmas \ref{Lem:defX} and \ref{Lem:defY} ensure that we are in the algebraic situation abstracted out in Lemma \ref{Lem:algfus}.

\begin{Thm}\label{Thm:fus} Let ${\mc Z}$ be a (local) smooth section of ${\mc L}^{\otimes c}$ over the extended Teichm\"uller space $\hat{\mc T}_2$ of marked surfaces of type $(\Sigma,X,Y,\tilde z,\tilde w,\dots)$. Assume that ${\mc Z}$ is $h_{r,s}$- (resp. $h_{2,1}$-) highest-weight w.r.t. $\tilde z$ (resp. $\tilde w$), satisfies the null vector equations 
\begin{align*}
\Delta^X_{r,s}{\mc Z}&=0\\
\Delta^Y_{2,1}{\mc Z}&=0
\end{align*}
and that for $\eps>0$ small enough
$${\mc Z}(\Sigma,X,Y,z,z-z(Y),\dots)=O((z(Y)-z(X))^{h_{r+1,s}-h_{r,s}-h_{2,1}-\eps})$$
as $Y\rightarrow X^+$. Assume furthermore that $\tau\notin\Q$ and $r\tau-s>0$. Then
$${\mc Z}_0(\Sigma,X,\tilde z,\dots)=\lim_{Y\rightarrow X^+}(z(Y)-z(X))^{-h_{r+1,s}+h_{r,s}+h_{2,1}}{\mc Z}(\Sigma,X,Y,z,z-z(Y),\dots)$$
is well-defined (locally in $\hat{\mc T}_1$), $h_{r+1,s}$-highest-weight, smooth, and satisfies
$$\Delta_{r+1,s}{\mc Z}_0=0$$
\end{Thm}
In the case $r\tau-s<0$, the statement holds by replacing $h_{r+1,s}$ with $h_{r-1,s}$ and $\Delta_{r+1,s}$ with $\Delta_{r-1,s}$. The condition $r\tau-s>0$ is used only to establish the existence and smoothness of ${\mc Z}_0$, and we expect it can be dispensed with (i.e., as $Y\rightarrow X$, ${\mc Z}$ is the superposition of two Frobenius series, and the leading term of both series satisfies a null-vector equation).
\begin{proof}
We begin by writing the condition $\Delta_{2,1}^Y{\mc Z}=0$ in coordinates and check it is of the type considered in Lemma \ref{Lem:mainreg}. 

As explained before Lemma \ref{Lem:loccomp}, we choose a 1-form regular and non-vanishing near $X$ depending smoothly on $(\Sigma,Z,Z',\dots$). We take $(w_1,\dots,w_n)$ smooth local coordinates on the Teichm\"uller space of surfaces of type $(\Sigma,Z,Z',\dots)$ and extend them to a system of local coordinates of type $(x,y,w_1,\dots,w_n)$ on the space of surfaces of type $(\Sigma,X,Y,Z,Z',\dots)$ by setting $z(.)=\int_Z^.\sigma$, $x=z(X)$, $y=z(Y)$. Note that this also gives reference local coordinates $z-z(X)$, $z-z(Y)$ at $X,Y$ respectively.

We then trivialize
$${\mc Z}(\Sigma,X,Y,\tilde z,\tilde w,\dots)=f(x,y,w_1,\dots,w_n)s_\zeta^c(\Sigma)\left(\frac{d\tilde z}{dz}(X)\right)^{h_{r,s}}\left(\frac{d\tilde w}{dz}(Y)\right)^{h_{2,1}}$$
where $s_\zeta^c$ is the reference section of ${\mc L}^{\otimes c}$. The condition $\Delta_{2,1}{\mc Z}=0$ translates into a hypoelliptic PDE for $f$ (see \cite{Dub_Virloc}, Section 5.4).

We have that 
$$L_{-1}^Y{\mc Z}={\frac{d}{dt}}_{|t=0}{\mc Z}(\Sigma,X,Y+t,\tilde z,\tilde w-t,\dots)=\left(\frac{d\tilde w}{dz}(Y)\right)^{-1}\left(\partial_yf+h_{2,1}f\partial_z\log\frac{d\tilde w}{dz}(Y)\right)s_\zeta^c(\Sigma)\left(\frac{d\tilde z}{dz}(X)\right)^{h_{r,s}}\left(\frac{d\tilde w}{dz}(Y)\right)^{h_{2,1}}
$$
and similarly for $(L_{-1}^Y)^2{\mc Z}$, so that we simply have
$$((L_{-1}^Y)^2{\mc Z})(\Sigma,X,Y,z-z(X),z,\dots)=(\partial_{yy}f) s_\zeta^c(\Sigma)$$
For $L_{-2}^Y$, by reasoning as in Lemma \ref{Lem:defY}, we find
$$(L_{-2}^Y{\mc Z})(\Sigma,X,Y,z-z(X),z,\dots)=\left(\left(\frac{1}{y-x}\partial_x+\frac{h_X}{(x-y)^2}+(reg)\right) f \right)s_\zeta^c(\Sigma)$$
where $(reg)$ is a degree 1 differential operator with smooth coefficients (including at $x=y$).

It follows that $f$ solves a second-order PDE of type:
$$(\partial_{yy} +\frac{\tau}{x-y}\partial_x-\frac{\tau h_{r,s}}{(y-x)^2}+(reg))f=0$$
which satisfies the H\"ormander bracket condition \eqref{eq:Hormcond}.

Remark that $h_{r+1,s}-h_{r-1,s}=r\tau-s$ is positive by assumption. By Lemma \ref{Lem:mainreg}, we conclude that 
$$f(x,y,w_1,\dots,w_n)=(y-x)^{\alpha_+}f_+(x,y,w_1,\dots,w_n)$$
where $\alpha_+=h_{r+1,s}-h_{r,s}-h_{2,1}$, where $f_+$ extends smoothly to the singular surface $x=y$. This justifies the asymptotic expansion \eqref{eq:asympexp}.

Next we translate the null vector equations in terms of the formal series
$$w=t^\alpha\sum_{k\geq 0}t^k{\mc Z}_k\in W\otimes V_{\alpha,h}$$
where $W$ denotes smooth sections on $\hat{\mc T}_1$.
 
{}From Lemma \ref{Lem:defX} we obtain $\hat\Delta_{r,s}w=0$; and from Lemma \ref{Lem:defY} we get $\tilde\Delta_{2,1}w=0$, with notation as in Lemma \ref{Lem:algfus}. Then we apply the said Lemma \ref{Lem:algfus} to conclude. 
  
\end{proof}

\section{Schramm's formula for multiple $\SLE$s}

In \cite{Sch_percform}, Schramm gives a closed form expression for the probability that the trace of a chordal $\SLE$ passes to the left of a marked point in the bulk. 

Here we consider $n$ commuting $\SLE$s (multiple $\SLE$s, $n$-leg $\SLE$s, or $n-\SLE$s) connecting the marked boundary points $X$ and $Y$ in a simply-connected domain $D$. When $\kappa\leq 4$, the $n$ traces are simple and mutually avoiding except at the endpoints (\cite{Dub_Comm}). They divide the domain $D$ into $n+1$ random sectors. A natural extension of the question addressed in \cite{Sch_percform} is to evaluate the probability that a marked bulk point $Z$ ends up in one these sectors (``watermelon probabilities").

Before turning to that problem, and to provide additional motivation to fusion, we discuss relevant situations in discrete models converging to multiple SLEs.

\subsection{Discrete models}\label{ssec:discrete}

\paragraph{Loop-Erased Random Walks and Uniform Spanning Trees.}
The easiest example is that of Loop-Erased Random Walks (LERWs), which are distributed as branches of the Uniform Spanning Tree (UST) by Wilson's algorithm \cite{Wi}. Consider a bounded, simply-connected domain $D$ with two marked boundary points $X,Y$ and a grid approximation $\Gamma_\delta\subset\delta\Z^2$ of $D$, and consider the Uniform Spanning Tree on this domain with wired boundary conditions. Condition on a branch from a point within $o(1)$ of $Y$ to join the boundary root within $o(1)$ of $X$. A celebrated result of Lawler-Schramm-Werner \cite{LSW_LERW} states that the distribution of such a branch converges, as $\delta\searrow 0$, to chordal $\SLE_2$ in $(D,X,Y)$. 

This can be modified \cite{Fomin,Dub_Comm,Dub_Euler,KozL} in order to accommodate multiple $\SLE$s in the following natural way. Consider $n$ points within $o(1)$ of $Y$ and the branches connecting them to the boundary root. Condition on these branches to exit the domain within $o(1)$ of $X$ and to be disjoint except within $o(1)$ of $X$. The resulting system of branches converges in distribution to a multiple $\SLE_2$. See Figure \ref{Fig:LERWmult}.
\begin{figure}
\begin{center}\includegraphics[scale=.8]{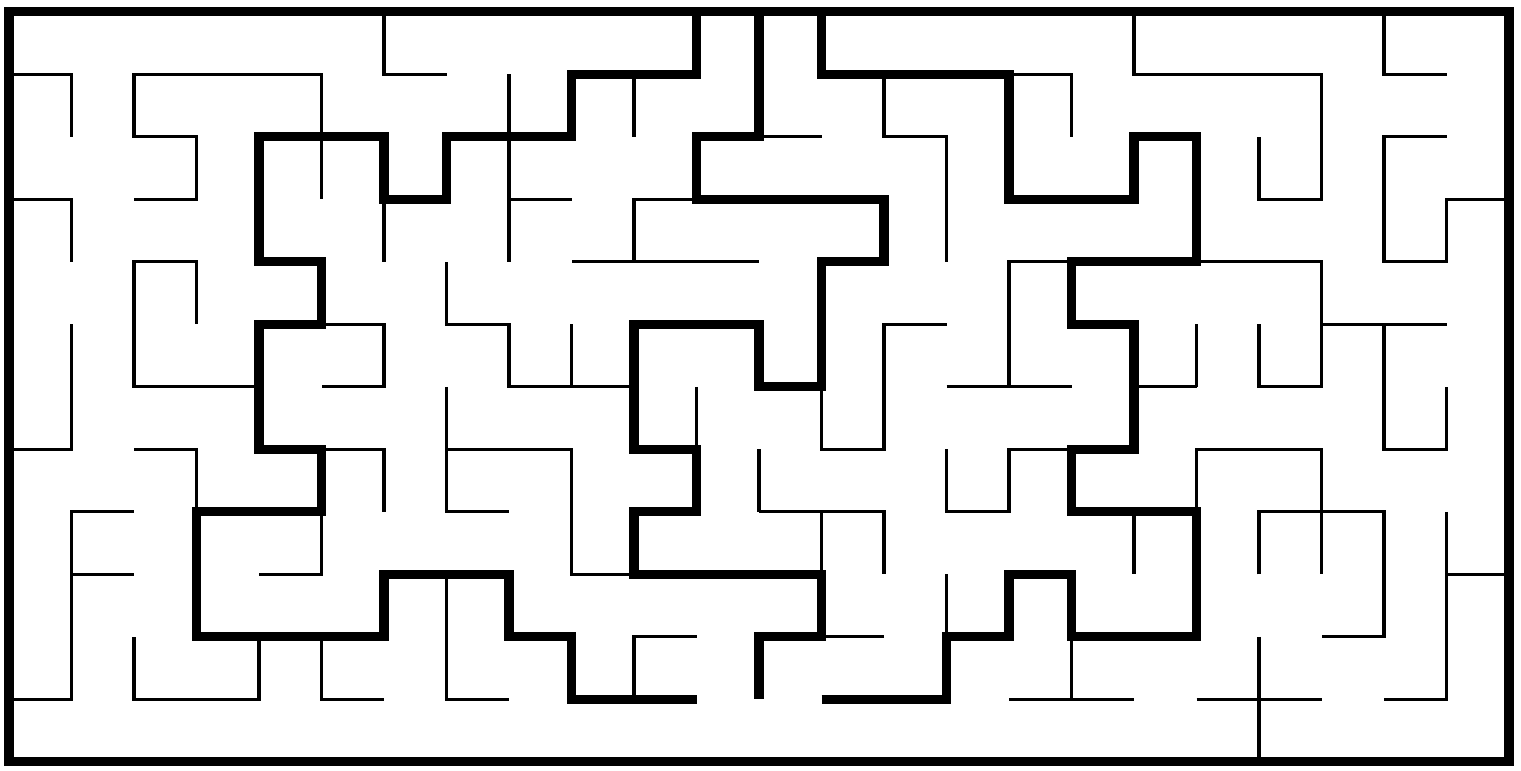}\end{center}
\caption{Uniform Spanning Tree with wired boundary conditions and three branches (bold) started near the bottom middle and conditioned to connect to the boundary near the top middle.}\label{Fig:LERWmult}
\end{figure}

Correspondingly, if $Z$ is a point in the bulk (the interior) of $D$, the probability that it has $k$ branches on one side and $n-k$ on the other side converges to $f_k(\theta)$, where $\theta$ is s.t. $(\H,0,\infty,e^{i\theta})$ is conformally equivalent to $(D,X,Y,Z)$. 

In percolation, or more generally FK percolation (or random-cluster model, see e.g. \cite{Grim_FK}), one can consider the following situation. In the grid approximation (with small mesh $\delta$) of a simply-connected domain $D$ with marked boundary points $X_1,\dots,X_n,Y_n,\dots,Y_1$ (listed in cyclic order), consider an FK configuration with alternating (at each marked boundary point) boundary conditions wired/free; this creates $n$ interfaces. Condition on the interface starting at $X_i$ to end at $Y_i$, $i=1,\dots,n$ and take the $X_i$'s (resp. $Y_i$'s) microscopically close to a point $X$ (resp. $Y$). Conjecturally, the system of interfaces converges to a multiple $\SLE$, and as before one can study e.g. their position relative to a marked bulk point $Z$.

In order to generate all possible null vectors (viz. with $r,s>1$), one needs to fuse commuting $\SLE$s with dual values of $\kappa$ (viz. in $\{\kappa,16/\kappa\}$) (it was observed in \cite{Dub_Comm} that commuting $\SLE$s have either the same or dual values of $\kappa$). Again one can describe a simple example in terms of UST and LERW.

Let $D$ be a simply-connected domain with two marked boundary points $X,Y$, and $\Gamma_\delta$ a subgraph of $\delta\Z^2$ approximating $D$. Consider $T$ a UST (sampled uniformly at random from spanning trees on $\Gamma_\delta$) with wired boundary conditions on the boundary arc $(XY)$ and free boundary conditions on the boundary arc $(YX)$. The dual tree $T^*$ is a UST on the dual graph, with boundary conditions swapped. In this context, Lawler, Schramm and Werner \cite{LSW_LERW} showed that the Peano exploration path (traced between $T$ and $T^*$) converges to chordal $\SLE_8$. One can modify this situation by taking a point $Y'_\delta$ on the dual graph within $O(\delta)$ of $Y$ and condition the corresponding branch of $T^*$ to reach the boundary within $O(\delta)$ of $X$. If $\gamma'$ denotes this branch and $\gamma$ is the Peano path on the side of that branch containing the wired boundary arc $(XY)$, the pair $(\gamma,\gamma')$ converges in law (in the small mesh limit) to a system of commuting $\SLE$s. See Figure \ref{Fig:UST22}.
\begin{figure}
\begin{center}\includegraphics[scale=.8]{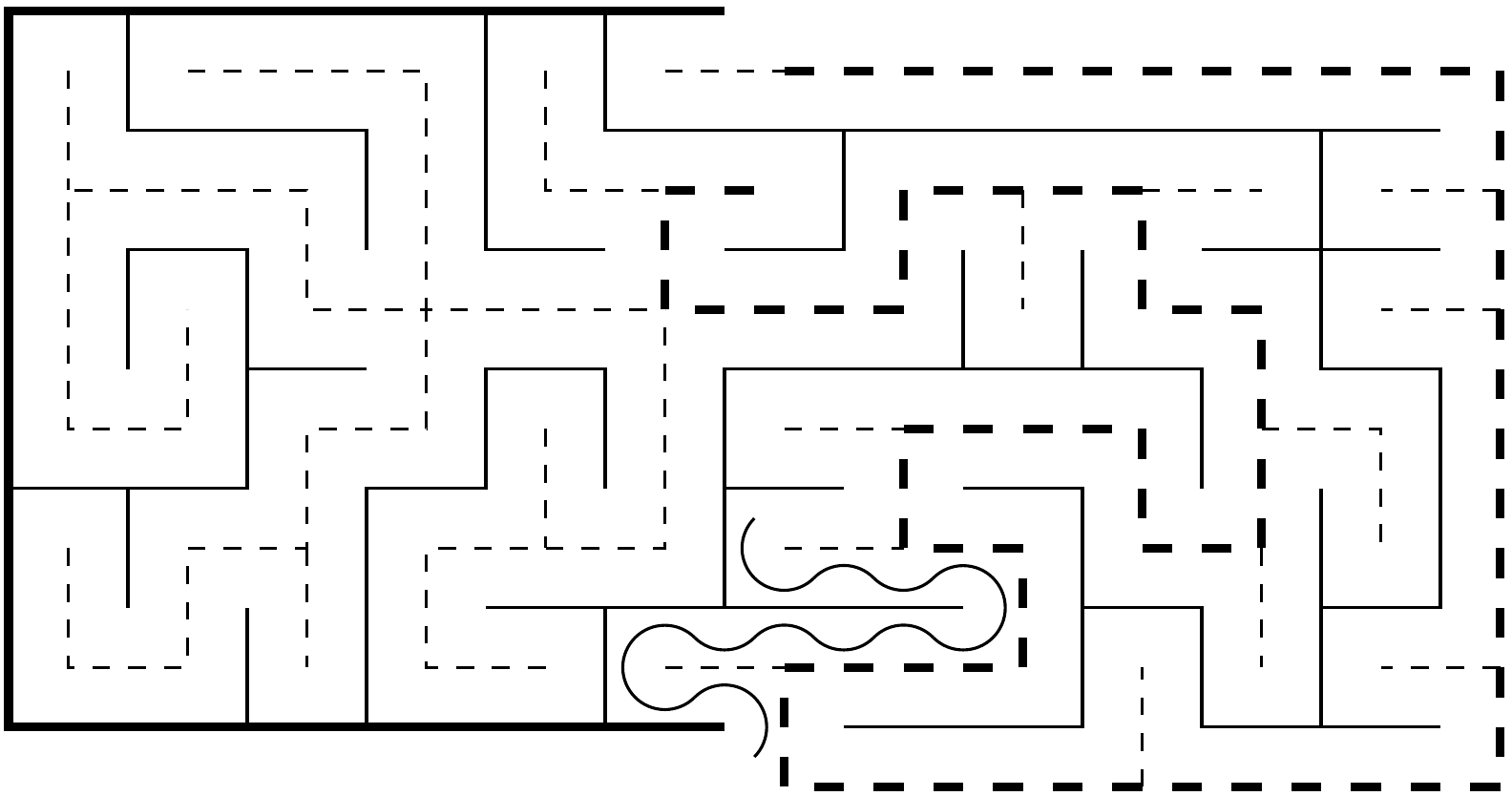}\end{center}
\caption{Uniform Spanning Tree (solid) with wired boundary condition on the left half and free on the right half. Dual tree (dashed). Dual branch started near top and condition to exit near bottom (bold dashed). Initial slit of the Peano exploration path (curved). 
}\label{Fig:UST22}
\end{figure}
A simple observable in this situation is the probability that a spectator point in the bulk lies on one side of the branch $\gamma'$. In the regular situation (when $X'_\delta$ is at macroscopic distance of $X,Y$), this is a ``martingale observable" w.r.t. an $\SLE_8$-type curve started at $X$ and an $\SLE_2$-type curve started at $X'$ (see e.g. the discussion in Section 2.5 of \cite{Dub_Comm}), leading to two second-order BPZ equations. We can then take the limit $X'\rightarrow X$ to obtain an observable with one fewer variable satisfying a higher-order BPZ equation.

\paragraph{Ising model.}
Let us briefly discuss the case of the Ising model, which is instructive and rather well understood.

In order to consider coupled primal and dual interfaces, a natural set-up is that of the Edwards-Sokal coupling of the Ising model with the random cluster model (see e.g. \cite{Grim_FK}). Fix an underlying planar graph $\Gamma=(V,E)$ and assign a positive weight $w(e)$ to each edge $e$. An admissible configuration is a pair $(\Gamma',\sigma)$ where $\Gamma'$ is a subgraph of $\Gamma$ and $\sigma\in\{\pm 1\}^V$ is a spin ``field" which is constant on connected components of $\Gamma'$. The weight of a configuration is $\prod_e w(e)$. Then $\sigma$ (resp. $\Gamma'$) is a sample of an Ising model (resp. random cluster model with $q=2$).

From the work of Smirnov \cite{Smi_ICM,CheSmi_univ}, we know that in the small mesh limit and for appropriate weighted graphs and boundary conditions, the boundaries of Fortuin-Kasteleyn (FK) clusters (connected component of $\Gamma'$) converge to $\SLE_{16/3}$-type curves and the spin interfaces (running between vertices with spin $+1$ on one side and $-1$ on the other) converge to $\SLE_3$-type curves.%

For instance consider the situation depicted in Figure \ref{Fig:ES2}.
\begin{figure}
\begin{center}\includegraphics[angle=90,scale=.4]{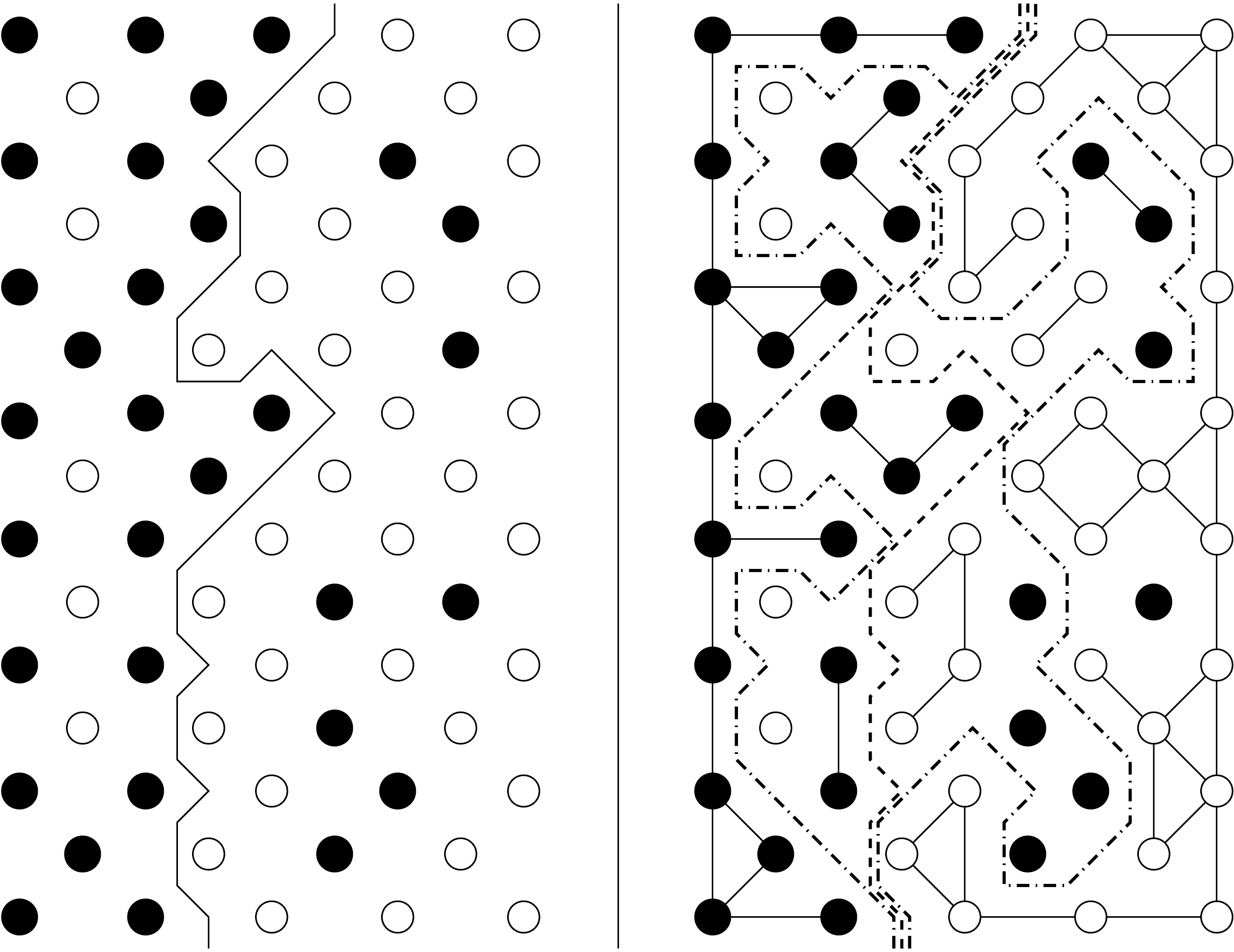}\end{center}
\caption{$+/-$ boundary conditions in the Edwards-Sokal coupling on the centered square lattice. Spins (black/white) are fixed on the boundary. Left panel: spin interface. Right panel: two FK interfaces (dash-dotted) - the boundaries of the two FK clusters attached to the boundary - and the spin interface (dashed).
}\label{Fig:ES2}
\end{figure}
The underlying graph is (a portion of) the centered square lattice $\Z^2\cup\left((\frac  12,\frac 12)+\Z^2\right)$ (with adjacency given by $(m,n)\sim (m\pm 1,n)$, $(m,n)\sim(m,n\pm 1)$, $(m,n)\sim (m\pm\frac 12,n\pm\frac 12)$ for $m,n\in\Z^2$); this is an isoradial triangulation. We can consider changes of boundary conditions from the spin or FK perspective.

In Figure \ref{Fig:ES2}, we consider a $+/-$ boundary condition change (see e.g. XI.3 in \cite{DiF} for a discussion of Ising boundary conditions from a CFT perspective). It gives rise to a single spin interface, with weight $h_{2,1}=\frac 12$; and to two FK interfaces, with weight $h_{1,3}=\frac 12$. 
The position of a spectator bulk point w.r.t. to the spin interfaces gives an observable satisfying the second-order BPZ equation corresponding to the singular vector $\Delta_{2,1}$ (Schramm's formula). Its position w.r.t. the FK interfaces satisfies the third-order BPZ equation derived from $\Delta_{1,3}$.

In Figure \ref{Fig:ES3}, we consider $+/-/$free boundary condition changes, as in \cite{HonKyt_free}.  
\begin{figure}
\begin{center}\includegraphics[angle=90,scale=.4]{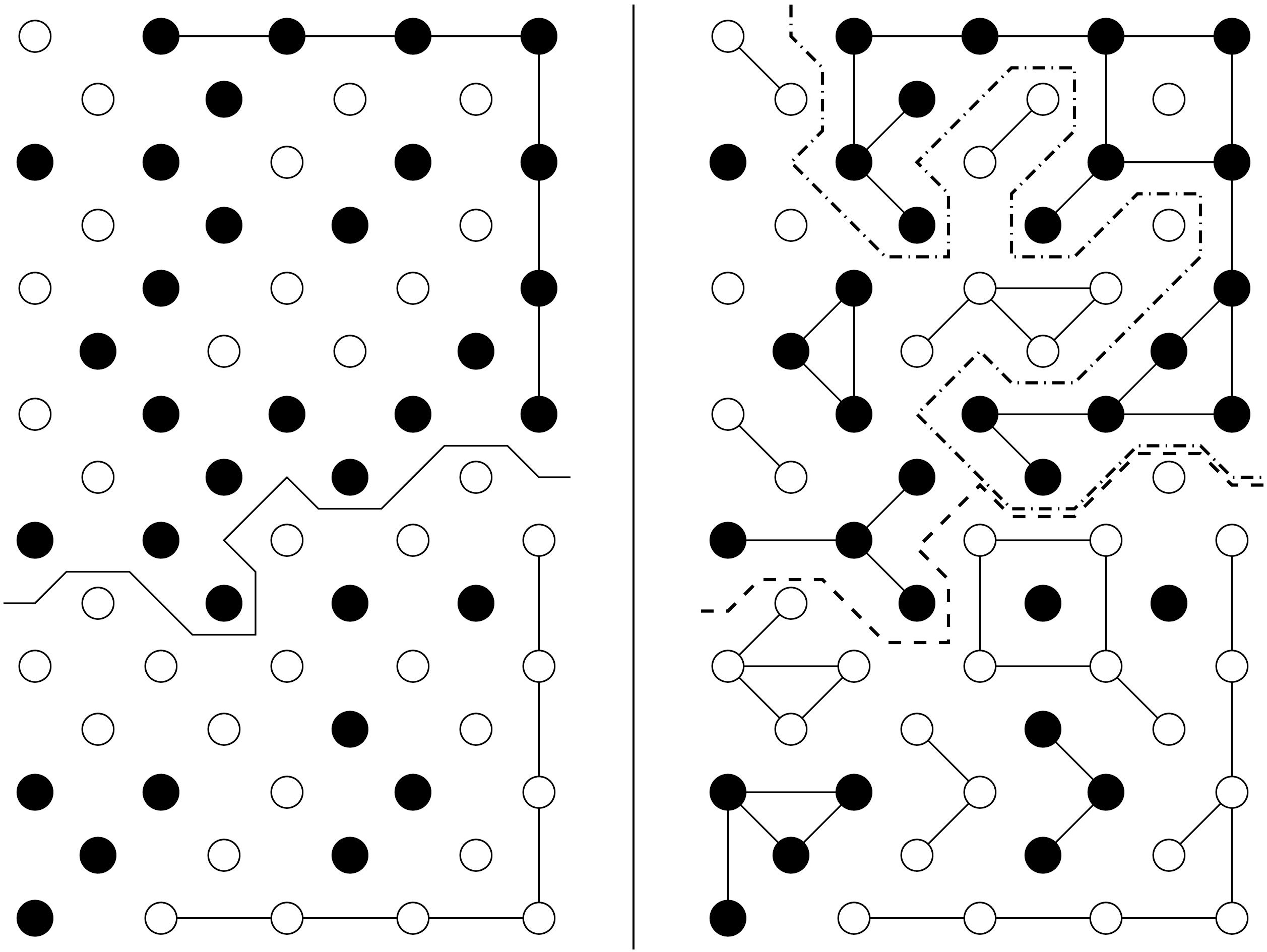}\end{center}
\caption{Boundary conditions: free (left); $+$ (top right); $-$ (bottom right). Left panel: spin interface. Right panel: FK interface (dash-dotted) and spin interface (dashed).
}\label{Fig:ES3}
\end{figure}
The $+$/free bcc (corresponding to an $\SLE_{16/3}$ variant) has weight $h_{1,2}=\frac 1{16}$, and as before the $+/-$ bcc has weight $h_{2,1}=\frac 12$. Upon fusing these two (viz. shrinking the $+$ boundary component to a point), one gets $h_{2,2}=\frac 1{16}=h_{1,2}$ (and indeed we are left with a $-$/free bcc). From the perspective of \cite{Dub_Comm}, this corresponds to the commutation of an $\SLE_3(-\frac 32,-\frac 32)$ with two $\SLE_{16/3}(-\frac 83,2)$'s. Notice that the coincidences of weights $h_{2,1}=h_{1,3}$, $h_{2,2}=h_{1,2}$ are specific to the Ising model (with central charge $c=1/2$).

\subsection{Differential equations for multiple SLE observables}

It will be technically convenient to use as a basic building block a system of $n$ $\SLE$s starting at $n$ distinct boundary points $X_1,\dots,X_n$ and converging to the same point $Y$ (rather than a chordal $\SLE$, as in \cite{Dub_Virloc}). This type of system was constructed in \cite{Dub_Comm}; we summarize their basic properties.

\begin{Prop}\label{Prop:nSLE}
Let $D$ denote a simply-connected domain and $X_1,\dots,X_n,Y$ distinct marked boundary points. For $\kappa\leq 4$, there is a measure $\mu_{(\Sigma,X_1,\dots,X_n,Y)}$ on $n$-tuples of simple paths $(\gamma_1,\dots,\gamma_n)$ disjoint in the bulk, where $\gamma_i$ has endpoints $X_i$ and $Y$, satisfying the following properties.
\begin{enumerate}
\item The partition function
\begin{equation}\label{eq:nlegpartfun}
{\mc Z}_{(D,X_1,\dots,X_n,Y)}=\|\mu_{(D,X_1,\dots,Y)}\|=\prod_{1\leq i<j\leq n}H_D(X_i,X_j)^{-1/\kappa}\prod_{i=1}^nH_D(X_i,Y)^{h_{2,1}+(n-1)/\kappa}
\end{equation}
has weight $h_{2,1}$ at $X_1,\dots,X_n$ and $h_{n+1,1}$ at $Y$.
\item (Conformal invariance) If $\phi: D\rightarrow D'$ is a conformal equivalence,
$$\phi_*\mu_{(D,X_1,\dots,Y)}=\mu_{(\phi(D),\phi(X_1),\dots,\phi(Y))}$$
\item (Restriction property) If $D'\subset D$ is a simply-connected domain agreeing with $D$ in neighborhoods of the marked points,
$$\mu_{(D',X_1,\dots,Y)}=\ind_{\gamma_1,\dots,\gamma_n\subset D'}\exp(\frac c2\nu_D(\cup_i\gamma_i;D\setminus D'))\mu_{(D,X_1,\dots,Y)}$$
\item (Marginals) Under $\mu^\sharp={\mc Z}^{-1}\mu$, the law of $\gamma_i$ is that of an $\SLE_\kappa(2,\dots,2)$ from $X_i$ to $Y$ with force points at $X_j$, $j\neq i$.
\item (Markov property) If $\tau$ is a stopping time for the filtration generated by $\gamma_i$ (running from $X_i$ to $Y$), then under $\mu^\sharp_{(D,X_1,\dots,Y)}$ the conditional law of $(\gamma_1,\dots,\gamma_n)$ given $\gamma_i^\tau$ is that of $(\gamma'_1,\dots,\gamma'_{i-1},\gamma_i^\tau\bullet\gamma'_i,\gamma'_{i+1},\dots)$ where $(\gamma'_1,\dots,\gamma'_n)$ have joint law $\mu^\sharp_{(D\setminus\gamma_i^\tau,X_1,\dots,X_{i-1},(\gamma_i)_\tau,X_{i+1},\dots,Y)}$.
\end{enumerate}
\end{Prop}
Here $\gamma^\tau$ denotes the path $\gamma$ stopped at time $\tau$ and $\bullet$ denotes the concatenation of paths; $H$ denotes the Poisson excursion kernel (see e.g. \cite{Dub_Virloc}), which induces a tensor dependence in local coordinates.
\begin{proof}
The probability measure $\mu^\sharp$ satisfying (and characterized by) 2,4,5 is constructed in \cite{Dub_Comm} (see in particular Section 8, also \cite{Dub_dual}). Define ${\mc Z}$ as in 1.; then 3. is a rephrasing of Lemma 8.1 in \cite{Dub_Comm}.
\end{proof}

From the conformal invariance and restriction properties we can construct an $n$-leg $\SLE$ on any bordered surface $\Sigma$ with marked boundary points $X_1,\dots,X_n,Y$ (and possibly additional markings), using the localization method explained e.g. in \cite{Dub_Virloc}. Then from the Markov property and reasoning again as in \cite{Dub_Virloc}, one obtains, under a local boundedness assumption on ${\mc Z}$, smoothness of ${\mc Z}$ and the null vector equations:
$$\Delta_{2,1}^{X_i}({\mc Z}s_\zeta^c)=0$$
where $\Delta_{2,1}^{X_i}$ refers to the Virasoro representation corresponding to a deformation at $X_i$, $i=1,\dots,n$, and $s_\zeta$ is the reference section of the determinant bundle.

We may also partition the configurations depending on the position of a marked bulk point $Z$ w.r.t. the paths (if we look at $Z$ as a puncture, this is simply decomposing the path space according to isotopy type). For $k=0,\dots,n$, we let ${\mc Z}_k$ be the partition function of systems of paths where $Z$ lies between $\gamma_k$ and $\gamma_{k+1}$ (where by convention $\gamma_0$ - resp. $\gamma_{n+1}$ - is the boundary arc between $X_1$ and $Y_1$ - resp. $X_n$ and $Y_n$). Then ${\mc Z}={\mc Z}_0+\cdots+{\mc Z}_n$. It follows that each ${\mc Z}_k$ is locally bounded (${\mc Z}$ is explicit \eqref{eq:nlegpartfun}) and consequently \cite{Dub_Virloc} is smooth and satisfies $\Delta^{X_i}_{2,1}({\mc Z}_ks_\zeta^c)=0$ at each marked boundary point. Trivially,
$$\mu^\sharp_{(D,X_1,\dots,X_n,Y)}\{Z{\rm\ between\ }\gamma_k{\rm\ and\ }\gamma_{k+1}\}=\frac{{\mc Z}_k}{{\mc Z}}(D,X_1,\dots,X_n,Y)$$

We are now in position to apply fusion, viz. to collapse $X_1,\dots,X_n$ in a single point $X=X_1$; this is done sequentially, first merging $X_2$ with $X_1$, and so on. For now we assume that $\kappa\notin\Q$. We assume $X_1,\dots,X_n,Y$ are in counterclockwise order on $\partial D$.

We have ${\mc Z}(D,X_1,X_2,\dots)=O(|X_2-X_1|^{2/\kappa})$ as $X_2\rightarrow X_1$, and {\em a fortiori} the same holds for ${\mc Z}_k$. Moreover $\Delta^{X_1}_{2,1}({\mc Z}_ks)=0$, $\Delta^{X_2}_{2,1}({\mc Z}_ks)=0$. From Theorem \ref{Thm:fus}, if $z$ is a local coordinate at $X_1$, we can expand
$${\mc Z}_k(D,X_1,X_2,\dots,z,z-z(X_2),\dots)=(z(X_2)-z(X_1))^{2/\kappa}{\mc Z}^{(2)}_k(D,X_1,X_3,\dots,z)(1+o(1))$$
as $X_2\rightarrow X_1$; then ${\mc Z}_k^{(2)}$ has weight $h_{3,1}$ at $X_1$ and satisfies
$$\Delta_{3,1}^{X_1}({\mc Z}^{(2)}_ks_\zeta^c)=0$$
at $X_1$. From Lemma \ref{Lem:defW}, we see that the null vector equations are preserved at $X_3,\dots$:
$$\Delta_{2,1}^{X_i}({\mc Z}^{(2)}_ks_\zeta^c)=0$$
for $i=3,\dots,n$. 

We also observe that ${\mc Z}_k^{(2)}\leq{\mc Z}^{(2)}$, which is itself explicit (a product of powers of Poisson excursion kernels). In particular we see
$${\mc Z}^{(2)}_k(X_1,X_3,\dots)=O(|X_3-X_1|^{h_{4,1}-h_{3,1}-h_{2,1}})$$
and we can apply again Theorem \ref{Thm:fus} to write
$${\mc Z}_k^{(2)}(D,X_1,X_3,\dots,z,z-z(X_3),\dots)=(z(X_3)-z(X_1))^{h_{4,1}-h_{3,1}-h_{2,1}}{\mc Z}^{(3)}_k(D,X_1,X_4,\dots,z)(1+o(1))$$
with
$$\Delta_{4,1}^{X_1}({\mc Z}^{(3)}_ks_\zeta^c)=0$$
Iterating the argument, we finally obtain a completely fused partition function ${\mc Z}_k^{(n)}(D,X,Y)$ satisfying
$$\Delta_{n+1,1}^{X}({\mc Z}^{(n)}_ks_\zeta^c)=0$$
and
$$\mu^\sharp_{(D,X,\dots,X,Y)}\{Z{\rm\ between\ }\gamma_k{\rm\ and\ }\gamma_{k+1}\}=\frac{{\mc Z}^{(n)}_k}{{\mc Z}^{(n)}}(D,X,Y,Z)$$
We check that ${\mc Z}^{(n)}\propto H_D^{h_{n+1,1}}s_\zeta^c$ (e.g. by evaluating in the upper half-plane with the standard local coordinate).

In the standard model $(D,X,Y,Z)=(\H,0,\infty,e^{i\theta})$, $\theta\in (0,\pi)$, we write
$$f_k(\theta)=\mu^\sharp_{(\H,0,\dots,0,\infty)}\{e^{i\theta}{\rm\ between\ }\gamma_k{\rm\ and\ }\gamma_{k+1}\}=\frac{{\mc Z}^{(n)}_k}{{\mc Z}^{(n)}}(\H,0,\infty,e^{i\theta})$$
We now want to translate the null vector equation for ${\mc Z}_k$ into an ODE for $f_k$. We proceed as in \cite{Dub_Virloc} (in particular Section 4.4), with a modification due to the fact that we have a spectator point $Z$ in the bulk (rather than several on the boundary).

Computationally it is slightly more convenient to consider
$$f_k(z)={\mc Z}_k(\H,0,\infty,z)$$
where the RHS is evaluated w.r.t. the standard local coordinates at $0$ and $\infty$ in $\H$, so that $f_k(\lambda z)=f_k(z)$ for $\lambda>0$. If we evaluate $L_m{\mc Z}_k$ in the same coordinates, we get $\ell_m^0f_k$, where
$$\ell_m^0=-\Re(z^{m+1})\partial_u-\Im(z^{m+1})\partial_v=-z^{m+1}\partial_z-{\bar z}^{m+1}\partial_{\bar z}=-r^{m+1}\cos(m\theta)\partial_r-r^m\sin(m\theta)\partial_\theta$$
where $z=u+iv=re^{i\theta}$, $\partial_z=\frac 12(\partial_u-i\partial_v)$, $\partial_{\bar z}=\frac 12(\partial_u+i\partial_v)$, $\partial_r=\cos(\theta)\partial_u+\sin(\theta)\partial_v$, $\partial_\theta=-r\sin(\theta)\partial_u+r\cos(\theta)\partial_v$. 

Although the polar coordinates are quite natural, in order to have rational coefficients one may set $t=\cot(\theta/2)$, so that
$$\cos(\theta)=\frac{t^2-1}{t^2+1},{\rm\  \  \ }\sin(\theta)=\frac{2t}{t^2+1},{\rm\  \  \ }\partial_\theta=-\frac{t^2+1}2\partial_t$$  
Then if $m<0$
$$\ell_m^0=-r^{m+1}T_{|m|}\left(\frac{t^2-1}{t^2+1}\right)\partial_r-r^mU_{|m|-1}\left(\frac{t^2-1}{t^2+1}\right)t\partial_t$$
where $T_p,U_p$ designate the Chebychev polynomials of the first and second kind, viz.
\begin{align*}
T_p(\cos(\theta))&=\cos(p\theta)\\
U_p(\cos(\theta))&=\frac{\sin((p+1)\theta)}{\sin(\theta)}
\end{align*}
for $p\geq 0$. We further substitute $s=-t^2$ (in terms of hypergeometric equations, this corresponds to a Goursat substitution) to obtain
$$\ell_m^0=-r^{m+1}T_{|m|}\left(\frac{s+1}{s-1}\right)\partial_r-r^mU_{|m|-1}\left(\frac{s+1}{s-1}\right)2s\partial_s$$
As explained in \cite{Dub_Virloc}, the null vector equation gives
$$\Delta_{n+1,1}^0f_k=0$$
where $\Delta_{n+1,1}^0$ (a differential operator) is obtained from $\Delta_{n+1,1}$ - an explicit element of ${\mc U}(\Vir^-)$, given by the Benoit--Saint-Aubin formula \eqref{eq:BSA} - by substituting $\ell^0_m$ for $L_m$. In the variables $r,s$, the coefficients of this differential operator are Laurent polynomials in $r,\kappa$ and rational in $s$ (with poles at $1$). 

By homogeneity we see that 
$$\Delta_{n+1,1}^0=r^{-n-1}{\mc D}_{n+1}+(\dots)\partial_r$$
where ${\mc D}_{n+1}\in\R[s,(s-1)^{-1}]\partial_s$ (or more precisely in $\Q[\kappa,\kappa^{-1},s,(s-1)^{-1}]\partial_s$). Since $f_k$ depends only on $s$ we get ${\mc D}_{n+1}f_k=0$.

We can also study the nature of singularities at $0,1,\infty$. For background on Fuchsian equations, see e.g. \cite{Yosh}.

\begin{Lem}
The differential operator ${\mc D}_{n+1}$ is Fuchsian (viz. has only regular singular points).
\end{Lem}
\begin{proof}
The only possible singularities are at $0,1,\infty$. Recall that
$$\sum_{i=0}^Np_i(s)\partial_s^i$$
is regular singular (or regular) at $a$ iff the coefficients $p_i$ are rational with a pole of order at most $N-i$ at $a$, and $p_N(a)\neq 0$.

At $0$ it is clear we can write
$${\mc D}_{n+1}=\sum_{i=0}^{n+1}P_i(s)(s\partial_s)^i$$
where $P_i$ a rational fraction regular at $0$; thus $0$ is a regular singular point. Note that the leading order term is $(-2s\partial_s)^{n+1}$, coming from the term $L_{-1}^{n+1}$ in $\Delta_{n+1,1}$.

We observe that in $\ell^0_m$, in both terms the order of the pole at 1 and the degree of $\partial_s$ add up to at most $-m$. Consequently $P_i$ has a pole of order $\leq n+1-i$ at $-1$, and $-1$ is a regular singularity.

At $\infty$, we substitute $s'=s^{-1}$ and get $s'\partial_{s'}=-s\partial_s$. Reasoning as for the singularity at $0$, we see that the singularity at infinity is regular. 
\end{proof}

In order to apply Theorem \ref{Thm:fus}, we used the (technical) condition $\kappa\notin\Q$ (from the discussion after the said Theorem, we already know this can be dispensed with if $n=2$). We are going to reason by density in $\kappa$, starting with the

\begin{Lem}\label{Lem:dens}
Let $\rho\geq 0$ and $\kappa_n\nearrow\kappa_\infty\leq 4$. Let $\gamma_n$ be (the trace of) an $\SLE_{\kappa_n}(\rho)$ from $(0,0^+)$ to $\infty$ in the upper half-plane $\H$. Let $\theta\in (0,\pi)$ and $\Theta_n\in (0,\pi]$ be: $\pi$ if $e^{i\theta}$ is to the left of $\gamma_n$; and otherwise s.t. 
$((\H\setminus\gamma)^r,0,\infty,e^{i\theta})$ is conformally equivalent to $(\H,0,\infty,e^{i\Theta})$, where $(\H\setminus\gamma)^r$ is the connected component of $\H\setminus\gamma$ to the right of $\gamma$. Then $\Theta_n$ converges in law to $\Theta_\infty$ as $n$ goes to infinity.
\end{Lem}
\begin{proof}
Let $\delta_n=1+2\frac{\rho+2}{\kappa_n}$, so that $\delta_n\searrow\delta\geq 2$. We can couple monotonically Bessel processes $(X^n_t)_{n\in\N,t\geq 0}$ of dimension $\delta_n$ (e.g. by the additivity property of Bessel processes), i.e. $t\mapsto X^n_t$ is a Bessel process of dimension $\delta_n$, and $X^n\searrow X^\infty$ uniformly on compact subsets a.s. 

In particular, $\int_0^tds/X^\infty_s$ is a.s. finite for finite $t$ and then
$\int_0^tds/X^n_s\nearrow\int_0^tds/X^\infty_s$ (a.s. uniformly in $t$ on compact time intervals).

From the $X^n$'s we can construct a sequence of coupled $\SLE_{\kappa_n}(\rho)$'s. Let $\theta_t^n$ be s.t. $(\H\setminus\gamma^n_{[0,t]},0,\infty,e^{i\theta})$ is conformally equivalent to $(\H,0,\infty,e^{i\theta_t})$. Then $\theta^n_t$ is a continuous functional of the processes $(X^n_{t'},\int_0^{t'} ds/X^n_s)_{t'\leq t}$. Consequently the law of $\theta_t^n$ converges weakly to that of $\theta^\infty_t$ for fixed $t$. 

We let $\Theta^n=\lim_{t\rightarrow\infty}\theta^n_t$ (in particular, $\Theta^n=\pi$ iff $e^{i\theta}$ is to the left of $\gamma^n$). Let $\tau$ be the first time $\gamma=\gamma^\infty$ exits $D(0,R)$, $R\gg 1$. By scaling, for $t\gg R^2$, with high probability $\tau\leq t$ and then $\gamma$ does not return to $D(0,\sqrt R)$ after $\tau$ (otherwise, again by scaling, $\gamma$ would be non-simple with positive probability).

If $n$ is large enough (depending on the realization of the coupling), $\gamma^n_\tau$ is outside of $D(0,R-1)$. By harmonic measure arguments (Beurling estimate), we see that $|\Theta^n-\theta_\tau^n|=O(R^{-1/4})$ unless $\gamma^n$ returns to $D(0,\sqrt R)$ after $\tau$. Hence w.h.p. $\theta^n_t$ is close to $\Theta^n$ uniformly in $n$, which concludes.

\end{proof}

We can now conclude.

\begin{Thm}\label{Thm:Schmult}
Let $\kappa\in (0,4]$ and $\gamma_1,\dots,\gamma_n$ be a multiple $\SLE$ from $0$ to infinity in $\H$. For $k=0,\dots,n$, let
$$f_k(\theta)=\P\{e^{i\theta}{\rm\ between\ }\gamma_k{\rm\ and\ }\gamma_{k+1}\}$$
Then $f_k$ satisfies ${\mc D}_{n+1}f_k=0$, which is Fuchsian differential equation of degree $n+1$ in the variable $s=-\cot^2(\theta/2)$.
\end{Thm}
This follows from \cite{Sch_percform} when $n=1$ and is discussed from a CFT perspective in \cite{CarGam_fus}; the case $c=0$, $n=2$ is treated in \cite{BelVik_bub}.
\begin{proof}
We just need to lift the condition $\kappa\notin\Q$. Remark that the coefficients of ${\mc D}_{n+1}$ are continuous (actually, Laurent polynomials) in $\kappa$. 

Then we argue that if $\kappa_j\nearrow\kappa$, then $f_k^{\kappa_j}(\theta)\rightarrow f_k(\theta)$. To see this, we start by sampling the leftmost strand $\gamma_1$, which is an $\SLE_\kappa(\rho)$ with $\rho=2(n-1)$. Then $((\H\setminus\gamma_1)^r,0,\infty,e^{i\theta})$ is equivalent to $(\H,0,\infty,e^{i\Theta_1})$, and by Lemma \ref{Lem:dens} the law of $\Theta_1$ converges as $j\rightarrow\infty$. We iterate and construct angles $(\Theta_1,\dots,\Theta_n)$, the joint distribution of which converges; here $\Theta_k$ is s.t. $((\H\setminus\gamma_k)^r,0,\infty,e^{i\theta})$ is equivalent to $(\H,0,\infty,e^{i\Theta_k})$. (Notice that the law of $\Theta_1$ is continuous in $\theta$, since $e^{i\theta}$ is a.s. at positive distance of  the $\gamma_i$'s). We have
$$\{e^{i\theta}{\rm\ between\ }\gamma_k{\rm\ and\ }\gamma_{k+1}\}
=\{\Theta_{k}<\pi,\Theta_{k+1}=\pi\}$$
which gives $f_k^{\kappa_j}(\theta)\rightarrow f_k(\theta)$.

Then $f_k$ is a simple limit of the $f_k^{\kappa_j}$, where we choose $\kappa_j$ to be irrational; and we have
$${\mc D}_{n+1}^{\kappa_j}f_k^{\kappa_j}=0$$
Since the coefficients of this linear differential operator (and their derivatives at any order) are continuous in $\kappa$ (choosing the leading coefficient to be $1$) and the $f_k$'s are bounded by $1$, it follows easily that they are equicontinuous, and so are their iterated derivatives. Consequently, up to extracting a subsequence we may assume that $f_k^{\kappa_j}$ converges in $C^{n+1}$ on compact subsets of $(0,\pi)$, and the limit $f_k$ is a solution of ${\mc D}_{n+1}f_k=0$.
\end{proof}

As explained after Proposition \ref{Prop:nSLE}, we can consider a more general situation, with $(\Sigma,X_1,\dots,X_n,Y,\dots)$ a bordered surface with marked boundary points $X_1,\dots,X_n,Y$ (and possibly additional markings). To this surface one associates by localization a measure $\mu_{(\Sigma,\dots)}$ on $n$-leg $\SLE$s started at the $X_i$'s and ending at $Y$ and the partition function ${\mc Z}(\Sigma,\dots)=\|\mu_{(\Sigma,\dots)}\|$. If $c\leq 0$, one can ensure finiteness by lifting to the the universal cover, as in the $n=1$ case \cite{Dub_Virloc}. Then we have the null vector equations $\Delta_{2,1}^{X_i}({\mc Zs})=0$, $i=1,\dots,n$. 

One naturally expects that the fused partition function ${\mc Z}^{(n)}(\Sigma,X,Y,\dots)$ satisfies $\Delta_{n+1,1}^{(n)}({\mc Z}^{(n)}s)=0$. Given Theorem \ref{Thm:fus}, the remaining (technical) difficulties consist of (a) ensuring {\em a priori} estimates of type
$${\mc Z}(X_1,X_2,\dots)=O(|z(X_2)-z(X_1)|^{h_{3,1}-2h_{2,1}-\eps})$$
and (b) lifting the $\kappa\notin\Q$ condition. The argument given for (a) in the special case of Theorem \ref{Thm:Schmult} extends {\em mutatis mutandis} to the situation where the $n$ strands are isotopic to the same simple path from $X$ to $Y$; in that case (and when $\kappa\notin\Q$) we get the null vector equation $\Delta_{n+1,1}({\mc Z}s)=0$.

For (b), we gave a direct argument after Lemma \ref{Lem:algfus} in the case $n=2$. We expect Lemma \ref{Lem:algfus} and consequently Theorem \ref{Thm:fus} to hold without this condition for general $n$.

Remark that, by \cite{Dub_Comm}, we may grow $\gamma_k$ alone, and its marginal law is that of an $\SLE_\kappa(\rho_l,\rho_r)$, where $(\rho_l,\rho_r)=2(k-1,n-k)$. For such an $\SLE$ with seed at $x$ and left (resp. right) force points at $0$ (resp. $1$), one may define 
$$g_k(x,z)=\P\{z{\rm\ to\ the\ left\ of\ }\gamma_k\}$$
Then $g_k$ satisfies a second-order PDE in three variables $x\in(0,1)$, $z\in\H$, with $g_k\rightarrow 1$ (resp. 0) when $z\rightarrow (-\infty,x)$ (resp. $(x,\infty)$). We also have $(f_1+\cdots+f_{k-1})(e^{i\theta})=\lim_{R\rightarrow\infty}g_k(x,Re^{i\theta})$. Hence in principle the $f_k$'s may be recovered sequentially by degenerating solutions of PDEs; it is however very unclear from this elementary argument that the $f_k$'s span the solution space of the (computable) Fuchsian ODE ${\mc D}_{n+1}$.

\section{Higher-order BPZ differential equations}

In this section we show how to generate a BPZ differential equation for arbitrary $(r+1,s+1)\in(\N^*)^2$ by collapsing $r+s$ second-order BPZ equations. Weights are parameterized as in \eqref{eq:param}. In a nutshell, the argument is to lift a system of $r+s$ BPZ equations to a representation of $r+s$ copies of the Virasoro algebra; operate fusion at that level, and go down to a higher-order BPZ equation in fewer variables. This gives another example where the machinery of Virasoro uniformization and representation theory produces a concrete analytic statement, and also shows how to obtain of the type of Theorem \ref{Thm:Schmult} in the presence of systems of commuting $\SLE$s with dual $\kappa$'s (as in Figures \ref{Fig:ES2} and \ref{Fig:ES3}), under suitable regularity assumptions.

We consider $(\H,x_1,\dots,x_r,y_1,\dots,y_s,z_1,\dots,z_n,\infty)$ the upper half-plane with $r+s+n+1$ marked boundary points. (Notice that we do not fix the translation and scaling degrees of freedom, which makes for simpler formulae). The $x$'s may be thought of a seeds of $\SLE$s, the $y$'s as seeds of dual $\SLE$s, and the $z$'s as spectator points. We associate weights to these marked points: $h_{2,1}$ for $x_1,\dots,x_r$, $h_{1,2}$ for $y_1,\dots,y_s$, and arbitrary weights $h_1,\dots,h_n$ for $z_1,\dots,z_n$. We also assume for definiteness
$$x_1<\cdots<x_r<y_1<\cdots<y_s<z_1<\cdots z_n.$$

Define differential operators $D_{2,1}^i$, $i=1,\dots,r$ and $D_{1,2}^j$, $j=1,\dots,s$ by
\begin{align*}
D_{2,1}^i&=\partial_{x_i}^2
+\tau\sum_{\substack{1\leq k\leq r\\k\neq i}}\left(\frac 1{x_k-x_i}\partial_{x_k}
-\frac{h_{2,1}}{(x_k-x_i)^2}
\right)
+\tau\sum_{\ell=1}^s\left(\frac{1}{y_\ell-x_i}\partial_{y_\ell}
-\frac{h_{1,2}}{(y_\ell-x_i)^2}
\right)
+\tau\sum_{k=1}^n\left(\frac{1}{z_k-x_i}\partial_{z_k}
-\frac{h_k}{(z_k-x_i)^2}
\right)
\\
D_{1,2}^j&=\partial_{y_j}^2
+\frac 1\tau\sum_{k=1}^r\left(\frac 1{x_k-y_j}\partial_{x_k}
-\frac{h_{2,1}}{(x_k-y_j)^2}
\right)
+\frac 1\tau\sum_{\substack{1\leq\ell\leq s\\ \ell\neq j}}\left(\frac{1}{y_\ell-y_j}\partial_{y_\ell}
-\frac{h_{1,2}}{(y_\ell-y_j)^2}
\right)
+\frac 1\tau\sum_{k=1}^n\left(\frac{1}{z_k-y_j}\partial_{z_k}
-\frac{h_k}{(z_k-y_j)^2}
\right)
\end{align*} 
We are interested in translation-invariant, homogeneous solutions of $D_{2,1}^iZ=D_{1,2}^jZ=0$ for all $i,j$. The degree of homogeneity depends on the (implicit) weight of the marked point at infinity. Notice that e.g. in $D_{2,1}^1$, the variables $x_2,\dots,y_s,z_1,\dots,z_n$ play the same role (and differ only by their weights).

An elementary positive solution is given by
$$Z_0=\prod_{1\leq i<j\leq r}(x_j-x_i)^{\frac\tau 2}\prod_{1\leq i<j\leq s}(x_j-x_i)^{\frac{\tau^{-1}} 2}\prod_{\substack{1\leq i\leq r\\ 1\leq j\leq s}}(y_j-x_i)^{-\frac 12}
\prod_{\substack{1\leq i\leq r\\ 1\leq k\leq n}}(z_k-x_i)^{a_k}
\prod_{\substack{1\leq j\leq s\\ 1\leq k\leq n}}(z_k-y_j)^{b_k}
\prod_{1\leq k<k'\leq n}(z_{k'}-z_k)^{2\tau b_kb_{k'}}
$$
where the $a_k$, $b_k$ solve
\begin{align*}
a_k(a_k-1)+\tau(a_k-h_k)&=0\\
b_k(b_k-1)+\tau^{-1}(b_k-h_k)&=0\\
a_k+\tau b_k&=0
\end{align*}
for $k=1,\dots,n$ (the first two equations are redundant given the third). In particular, we can take all the $a,$'s, $b$'s and $h$'s to be zero. When evaluating e.g. $Z_0^{-1}D_{2,1}^1Z_0$, one checks that all the double poles are removable, and the other terms cancel out thanks to the identity
$$\frac{1}{(u-v)(v-w)}+\frac{1}{(v-w)(w-u)}+\frac{1}{(w-u)(u-v)}$$
Remark that
\begin{align*}
\frac\tau 2&=h_{3,1}-2h_{2,1}\\
\frac{\tau^{-1}}2&=h_{1,3}-2h_{1,2}\\
-\frac 12&=h_{2,2}-h_{1,2}-h_{2,1}
\end{align*}
We can also consider the conjugate operators
\begin{align*}
{\mc L}^i=Z_0^{-1}D_{2,1}^iZ_0&=
\partial_{x_i}^2+\tau\sum_{1\leq k\leq r,k\neq i}\left(\frac 1{x_k-x_i}\right)
(\partial_{x_k}-\partial_{x_i})
+\sum_{\ell=1}^s\left(\frac{\tau}{y_\ell-x_i}\partial_{y_\ell}
-\frac{1}{x_i-y_\ell}\partial_{x_i}
\right)\\
&+\sum_{k=1}^n\left(\frac{\tau}{z_k-x_i}\partial_{z_k}
+\frac{2a_k}{x_i-z_k}\partial_{x_i}
\right)\\
{\mc M}^j=Z_0^{-1}D_{1,2}^jZ_0&=
\partial_{y_j}^2+\tau^{-1}\sum_{1\leq \ell\leq s,\ell\neq j}\left(\frac 1{y_\ell-y_j}\right)
(\partial_{y_\ell}-\partial_{y_j})
+\sum_{k=1}^r\left(\frac{\tau^{-1}}{x_k-y_j}\partial_{x_k}
-\frac{1}{y_j-x_k}\partial_{y_j}
\right)\\
&+\sum_{k=1}^n\left(\frac{\tau^{-1}}{z_k-y_j}\partial_{z_k}
+\frac{2b_k}{y_j-z_k}\partial_{y_j}
\right)
\end{align*}
which are (up to multiplicative constant) the generators of a system of $\SLE$s satisfying local commutation in the sense of \cite{Dub_Comm}.

Setting $Z=fZ_0$, we are now interested in homogeneous, translation invariant solutions of
\begin{equation}\label{eq:commmart}
{\mc L}^if={\mc M}^jf=0
\end{equation}
for $i=1,\dots,r$, $j=1,\dots,s$.

If $\hat{\mc T}$ is the extended Teichm\"uller space of simply-connected domains with $r+s+n+1$ marked boundary points (with formal local coordinates marked at the $r+s$ seeds $X_1,\dots,Y_s$ and 1-jets marked at the $n+1$ spectator points), one can define the function ${\mc Z}$ s.t. 
$${\mc Z}(\H,x_1,\dots,z_n,\infty)=Z(x_1,\dots,z_n)$$
where on the LHS the formal local coordinates and jets are given by the standard local coordinates in $\H$ (viz. $z-x$ at $x\in\R$ and $-z^{-1}$ at infinity). Then (see e.g. \cite{Dub_Virloc}, Section 4.4) we have
$$\Delta_{2,1}^{X_i}({\mc Z}s_\zeta^c)=\Delta_{1,2}^{Y_j}({\mc Z}s_\zeta^c)=0$$
where $s_\zeta^c$ is again the reference section of the determinant bundle ${\mc L}^{\otimes c}$.

Assume that $f$ satisfies \eqref{eq:commmart} and is bounded as $x_1,\dots,y_s\rightarrow 0$, while $z_1,\dots,z_n$ stay bounded and bounded away from each other and $0$. Assume that 
$$g(z_1,\dots,z_n)=\lim_{y_s\rightarrow 0}\cdots\lim_{x_1\rightarrow 0}f(x_1,\dots,y_s,z_1,\dots,z_n)$$
is well-defined and smooth (we make this additional assumption since the condition $r'\tau-s'>0$ of Theorem \ref{Thm:fus} is not satisfied for all fusions here). Then by repeated applications of Theorem \ref{Thm:fus} (under the assumption $\tau\notin\Q$), we obtain that
$g$ satisfies the BPZ equation
$$D_{r+1,s+1}(g{\overline Z}_0)=0$$
where $D_{r+1,s+1}$ is obtained from the singular vector $\Delta_{r+1,s+1}\in{\mc U}(\Vir^-)$ by substituting 
$$\ell_n^0=\sum_{k=1}^n(-z_k^{n+1}\partial_{z_k}-h_k(n+1)z_k^{n})$$
for $L_n$, and ${\overline Z}_0$ is the leading coefficient of $Z_0$ under fusion, viz.
$${\overline Z}_0=\prod_k z_k^{(s-r\tau)b_k}\prod_{k<k'}(z_{k'}-z_k)^{2\tau b_kb_{k'}}$$
Notice that the statement is not immediate even in the case $f\equiv 1$, $g\equiv 1$. 
 
At this stage the difficulty is that it is not clear how to construct such an $f$ for general $r,s$ or how to interpret it in terms of systems of $\SLE$s

{\bf Acknowledgments}. It is my pleasure to thank Denis Bernard, Cl\'ement Hongler and Fredrik Johansson Viklund for interesting conversations during the preparation of this article. I also thank anonymous referees for many useful comments.

\bibliographystyle{abbrv}
\bibliography{biblio}

-----------------------

\noindent Columbia University\\
Department of Mathematics\\
2990 Broadway\\
New York, NY 10027

\end{document}